\documentclass[10pt,reqno]{amsart}

\usepackage[b5paper,top=1.2in,left=0.9in]{geometry}
\usepackage{verbatim}
\usepackage{amssymb}
\usepackage{amsmath}
\usepackage{graphicx}
\usepackage{color}
\usepackage{amsthm}
\usepackage{tikz}
\usetikzlibrary{arrows,positioning,decorations.pathmorphing,  decorations.markings}

\renewcommand{\d}{\mathrm{d}}
\newcommand{\D}{\mathrm{D}}
\newcommand{\e}{\mathrm{e}}

\newtheorem{Thm}{Theorem}[section]
\newtheorem{Lem}[Thm]{Lemma}
\newtheorem{Prop}[Thm]{Proposition}
\newtheorem{Cor}[Thm]{Corollary}
\newtheorem{Rem}[Thm]{Remark}
\newtheorem{Def}[Thm]{Definition}

\theoremstyle{definition}
  \newtheorem*{proof1}{Proof of Theorem \ref{main_lift}}
  
  \theoremstyle{definition}
  \newtheorem*{proof2}{Proof of Theorem \ref{main_group}}

\newtheoremstyle{named}{}{}{\itshape}{}{\bfseries}{.}{.5em}{#1 #3}
\theoremstyle{named}

\def\R{\mathbb{R}}

\def\C{\mathbb{C}}
\def\Z{\mathbb{Z}}

\def\g{\mathfrak{g}}

\def\sl{\mathfrak{sl}}
\def\gl{\mathfrak{gl}}

\def\SL{\til{SL}}

\def\cD{\mathcal{D}}

\def\cL{\mathcal{L}}

\def\cN{\mathcal{N}}
\def\cO{\mathcal{O}}

\def\cQ{\mathcal{Q}}

\def\cS{\mathcal{S}}

\def\cU{\mathcal{U}}
\def\cV{\mathcal{V}}

\def\cZ{\mathcal{Z}}

\def\a{\alpha}
\def\b{\beta}
\def\c{\gamma}

\def\D{\Delta}

\def\d{\delta}
\def\e{\epsilon}

\def\h{\theta}

\def\l{\lambda}

\def\s{\sigma}
\def\t{\tau}

\def\Up{\Upsilon}

\def\w{\omega}

\def\bD{\mathbb{D}}

\def\bo{\textbf{o}}

\def\=>{\Longrightarrow}
\def\inj{\hookrightarrow}
\def\corr{\longleftrightarrow}

\def\to{\longrightarrow}
\def\tto{\longleftarrow}

\def\ox{\otimes}
\def\o+{\oplus}
\def\bo+{\bigoplus}
\def\x{\times}
\def\<{\langle}
\def\>{\rangle}
\def\({\left(}
\def\){\right)}
\def\oo{\infty}

\def\^{\wedge}
\def\+{\dagger}

\def\sub{\subset}
\def\sign{\mathrm{sign}}
\def\inv{^{-1}}

\def\half{\frac{1}{2}}

\def\dd[#1,#2]{\frac{d#1}{d#2}}
\def\del[#1,#2]{\frac{\partial #1}{\partial #2}}
\def\over[#1]{\overline{#1}}
\def\vec[#1]{\overrightarrow{#1}}

\def\tab{\;\;\;\;\;\;}
\def\dia{\diamondsuit}

\newcommand{\til}[1]{\widetilde{#1}}
\newcommand{\what}[1]{\widehat{#1}}
\newcommand{\xto}[1]{\xrightarrow{#1}}

\newcommand{\veca}[2][cccccccccccccccccccccccccccccccccccccccccc]{\left(\begin{array}{#1}#2 \\ \end{array} \right)}

\newcommand{\case}[2][cccccccccccccccccccccccccccccccccccccccccc]{\left\{\begin{array}{#1}#2 \\ \end{array}\right.}
\newcommand{\Eq}[1]{\begin{align}#1\end{align}}
\newcommand{\Eqn}[1]{\begin{align*}#1\end{align*}}

\begin{document}
\title{ $\cN=2$ Super-Teichm\"uller Theory}

\author{Ivan C. H. Ip}
\address[Ivan C. H. Ip]{Department of Mathematics, Hong Kong University of Science and Technology (HKUST), Hong Kong}
\email{ivan.ip@ust.hk}
\urladdr{http://www.math.ust.hk/~ivanip}

\author{Robert C. Penner}
\address[R. C. Penner]{Institut des Hautes \' Etudes Scientifiques, Bur-sur-Yvette, France; Department of Mathematics, UCLA, 
Los Angeles, CA}
\email{rpenner@math.ucla.edu, rpenner@ihes.fr}

\author{Anton M. Zeitlin}
\address[Anton M. Zeitlin]{Department of Mathematics, Louisiana State University, Baton Rouge, USA; 
IPME RAS, St. Petersburg}
\email{zeitlin@lsu.edu}
\urladdr{http://math.lsu.edu/~zeitlin}
\date{\today}

\numberwithin{equation}{section}

\maketitle

\begin{abstract}
Based on earlier work of the latter two named authors on the higher super-Teichm\"uller space with $\cN=1$, a component
of the flat $OSp(1|2)$
connections on a punctured surface, here we extend to the case $\cN=2$ of flat $OSp(2|2)$ connections.
Indeed, we construct here coordinates on the higher super-Teichm\"uller space of a surface $F$ with at least one puncture associated to the supergroup $OSp(2|2)$, which in particular specializes to give another treatment for $\mathcal{N}=1$ which is simpler than the earlier work.   The Minkowski space in the current case, where the corresponding super Fuchsian groups act, is replaced by the superspace $\mathbb{R}^{2,2|4}$, and the familiar lambda lengths are extended by odd invariants of triples of special isotropic vectors in  $\mathbb{R}^{2,2|4}$ as well as extra bosonic parameters, which we call ratios, defining a flat $\mathbb{R}_{+}$-connection on $F$. 
As in the pure bosonic or $\cN=1$ cases, we derive the analogue of Ptolemy transformations for all these new variables.
\end{abstract}

\tableofcontents

\newpage
%==============================================================================
\section{Introduction}\label{sec:intro}
%==============================================================================

\subsection{Brief overview of Teichm\"uller and super-Teichm\"uller theory}

The three cases $\cN=0,1,2$, where $\cN$ is half the number of odd generators, correspond to the respective Lie (super)  groups $G=PSL_2({\mathbb R}), OSp(1|2)$ and $OSp(2|2)$, with $PSL_2({\mathbb R})$ denoting the M\"obius group and OSp the orthosymplectic groups \cite{kac,dict}. The (super) Teichm\"uller space is defined as
a component of the moduli space
${\rm Hom}(\pi_1(F),G)/G$
of flat $G$-connections on the surface $F$, namely,
\Eqn{T_G(F):={\rm Hom}'(\pi_1(F),G)/G,\label{TF}}
where $G$ acts naturally by conjugation and
${\rm Hom}'$ denotes those
faithful (i.e., injective) representations $\rho:\pi_1(F)\to G$
onto a discrete subgroup of the Lie (super) group $G$
(i.e., the identity of $G$ is isolated in $\rho(\pi_1(F))$)
that satisfy the further condition that the value of $\rho$ on a 
loop that becomes null homotopic upon replacing a puncture
must be a parabolic element (i.e., have trace $\pm2$) in the case of $G=PSL_2({\mathbb R})$ and project to such a
parabolic Fuchsian transformation in case $G=OSp(1|2)$ or $OSp(2|2)$\footnote{This condition requires the image of $\pi_1(F)$ in the $\cN=2$ case to belong to the subgroup of $OSp(2|2)$ with the same Lie algebra, which we denote as $\widetilde{SL}(1|2)$. The reason for this notation is that this  subgroup is isomorphic to the semidirect product of a certain involution on the Lie algebra 
$\mathfrak{sl}(1|2)\simeq\mathfrak{osp}(2|2)$ and the connected component of identity of supergroup $SL(1|2)$   (for more details see Appendix \ref{sec:SL12}). }.

In fact, it was a useful innovation already in the bosonic case $\cN=0$ in \cite{P1} to require at least one puncture and to work in Minkowski space ${\mathbb R}^{2,1}$,
where $PSL_2({\mathbb R})\approx SO_+(2,1)$ agrees with the component of the identity in the group of Minkowski isometries.
Upon identifying the open positive light-cone with the space of all horocycles, the $\lambda$-length (i.e., the square root of the inner product of two isotropic vectors) is precisely the square root of the exponential of signed hyperbolic lengths between corresponding horocycles.  In case $\cN=1$, the analogous action of $OSp(1|2)$ as a subgroup of the Minkowski isometries $OSp(2,1|2)$ acting on ${\mathbb R}^{2,1|2}$ was exploited in \cite{PZ}.  Here we
study an $OSp(2|2)$ subgroup of $OSp(2,2|4)$ acting on 
the superspace ${\mathbb R}^{2,2|4}$.
An ideal triangulation $\Delta$ of $F$ (i.e., a collection of disjointly embedded arcs 
decomposing $F$ into triangles with vertices at the punctures) provides the convenient 
index set for the collection of $\lambda$-lengths with different choices of ideal triangulation
playing the role of different bases for the coordinate system; flips, (i.e., remove an edge $e$ of $\Delta$ to produce a complementary  quadrilateral with frontier $a,b,c,d$, and replace with the other diagonal $f$ of this quadrilateral) provide the basic transformations
relating ideal triangulations with the bosonic lambda lengths governed by the Ptolemy relation $ef=ac+bd$.  Dual to $\Delta$ (with one trivalent vertex
for each triangle in $F-\cup\Delta$ and one edge connecting each pair of vertices whose corresponding triangles
share an edge), there is a fatgraph $\tau$ embedded in $F$.
Part of the utility of $\lambda$-lengths comes from their role in a recursive construction of the so-called lift 
$\ell: \widetilde\Delta_\infty\to{\mathcal L}_0$, where $\widetilde\Delta_\infty$ denotes the set of ideal
points at infinity of a lift of $\Delta$ to the universal cover of $F$ and ${\mathcal L}_0$ denotes an appropriate set of isotropic vectors in Minkowksi space.  This paradigmatic role in case $\cN=0$ carries over
to $\cN=1,2$ albeit in the more complicated circumstance of dependence on spin structures on $F$; indeed on $\cN$ many
spin structures insofar as a component of super-Teichmueller space $T_{OSp(1|2)}(F)$  is determined by a spin structure and a component of super-Teichmueller space $T_{OSp(2|2)}(F)$ by two independent spin structures, see Section 1.2.  and Theorem \ref{compn}.

In each case of $\cN=0,1,2$, our coordinates are defined on a ${\mathbb R}_+^s$-bundle, the decorated Teichm\"uller space
$\til{T}(F)$ over $T(F)$, which amounts to the set of all possible Fuchsian lifts. Here $s>0$ is the number of punctures of the surface $F$.  It is only in the bosonic case $\cN=0$ that these extra parameters, the decorations, admit an explicit geometric interpretation as $s$-tuples of lengths of horocycles, one horocycle about each puncture.
However, in all three case $\cN=0,1,2$,
one can pass from $\lambda$-lengths to cross ratios $\mathcal{X}=\frac{ac}{bd}$ on each edge of $\Delta$ to describe coordinates on the undecorated (super) Teichm\"uller space itself, where the product of cross ratios
about each puncture must equal unity.  There is also the notion of ``surfaces with holes", where one drops
this condition on cross ratios and the punctures can open to oriented boundary components, cf. \cite{P2}.
  
In this paper, we study in detail the case of $\cN=2$, which is important on its own since it is known that $\cN=2$ super-Riemann surfaces are dual to $(1|1)$-dimensional complex supermanifolds. Thus the super-Fuchsian group we obtain correspond to a very important class of $(1|1)$-supermanifolds, namely the ones which could be obtained by uniformization. Also, from a certain point of view in physics \cite{drs}, $\cN=2$ supermoduli space may be even more important than $\cN=1$. To understand the main improvements obtained in this paper, let us first review the construction of \cite{PZ} in the case of $\cN=1$. In the remainder of the paper, we shall work over the Grassmann algebra $\cS_\R$ (see Appendix \ref{sec:SL12}).

\subsection{Review of the main results for $\cN=1$} 
In \cite{PZ} we studied the super-Teichm\"uller space ${T}_{OSp(1|2)}(F)$ defined in \eqref{TF}, where the surface $F$ has at least one puncture and negative Euler characteristic. The image of the fundamental group produces a generalization of the standard Fuchsian group $\Gamma$, which acts on a super-analogue of the upper half-plane $H^+$ producing $\cN=1$ super-Riemann surfaces as a factor $H^+/\Gamma$ (cf. \cite{crabin}).

The goal of the paper \cite{PZ} was to construct the analogue of Penner coordinates on 
the decorated $\cN=1$ super-Teichm\"uller space $\til{T}_{OSp(1|2)}(F)$, i.e. on $\mathbb{R}^s_+$-bundle 
$\til{{T}}_{OSp(1|2)}(F)=\mathbb{R}^s_+\times{T}_{OSp(1|2)}(F)$. The difficulty in constructing those coordinates is that ${{T}}_{OSp(1|2)}(F)$ (unlike the standard  ${T}_{PSL_2(\mathbb{R})}(F)$)
has many connected components enumerated by spin structures, which is due to the fact that 
the pure even subgroup of $OSp(1|2)$ is $SL_2(\mathbb{R})$ instead of 
$PSL_2(\mathbb{R})$. Thus, to proceed further it is necessary to have a suitable combinatorial description of the spin structures.

In \cite{PZ} we described three equivalent descriptions of spin structures on $F$:
\begin{itemize}
\item $\mathbb{Z}_2$-valued quadratic forms on $H_1(F,\mathbb{Z}_2)$-homology with respect to the standard intersection pairing;
\item  Lifts $\til{\rho}:\pi_1(F)\to SL(2,\R)$ of the Fuchsian representation $\rho:\pi_1(F)\to PSL(2,\R)$;
\item Classes of orientations on fatgraphs of $F$.
\end{itemize}
While the equivalence between first two was known, see \cite{Nat}, 
the third one, which is purely combinatorial, was novel and crucial for the main result of \cite{PZ}. 

Let us describe it in more detail: Consider the trivalent {\it fatgraph} $\tau$ (i.e., a fatgraph all of whose vertices have valence three), corresponding to the Riemann surface $F$, which is homotopy equivalent to $F$ with cyclic orderings on half-edges for every vertex \cite{P2} induced by the orientation of the surface. 
As discussed in the introduction, there is a one-to-one correspondence between ideal triangulations and trivalent fatgraphs induced by Poincar\'e duality in the surface. Let $\w$ be an orientation on the edges of $\t$. One can define a \emph{fatgraph reflection} at a vertex $v$ of $(\t,\w)$ as a reversal of the orientations of $\w$ on every edge of $\t$ incident to $v$. 

Let $\mathcal{O}(\t)$ be the equivalence classes of orientations on a trivalent fatgraph $\tau$ of $F$, where $\omega_1\sim \omega_2$ if and only if $\omega_1$ and $\omega_2$ differ by a finite number of fatgraph reflections. In \cite{PZ}, we identified such classes of orientations on fatgraphs with the spin structures on $F$ described via quadratic forms and lifts. 
In this paper we investigate this relation even further, deriving explicit formulas for the quadratic form on every fatgraph cycle.  
For example, the paths corresponding to the boundary cycles on the fatgraph (i.e., the punctures of $F$) are divided into two classes depending on the parity of the number $k$  of edges with orientation opposite to the canonical orientation of $\gamma$. These punctures are called  Ramond (R) when $k$ is even, and Neveu-Schwarz (NS) when $k$ is odd,  based on the value of the corresponding quadratic form.  

In \cite{PZ}, we also described how orientations change under the flip transformations. In this paper we present a unifying single-picture description, see  Figure \ref{flipgraphint}, 
\begin{figure}[h!]
\begin{center}

\begin{tikzpicture}[ultra thick, baseline=1cm]
\draw (0,0)--(210:1) node[above] at (210:0.7){$\e_2$};
\draw (0,0)--(330:1) node[above] at (330:0.7){$\e_4$};
    % draw the connecting line
    \draw[ 
 	ultra thick,
        decoration={markings, mark=at position 0.5 with {\arrow{>}}},
        postaction={decorate}
        ]
        (0,0) -- (0,2);
\draw[yshift=2cm] (0,0)--(30:1) node[below] at (30:0.7){$\e_3$};
\draw[yshift=2cm] (0,0)--(150:1) node[below] at (150:0.7){$\e_1$};
\end{tikzpicture}
\begin{tikzpicture}[baseline]
\draw[->, thick](0,0)--(2,0);
\node[above] at (1,0) {};
\node at (-1,0){};
\node at (3,0){};
\end{tikzpicture}
\begin{tikzpicture}[ultra thick, baseline]
\draw (0,0)--(120:1) node[above] at (100:0.3){$\e_1$};
\draw (0,0)--(240:1) node[below] at (260:0.3){$\e_2$};
    % draw the connecting line
    \draw[ 
        decoration={markings, mark=at position 0.5 with {\arrow{<}}},
        postaction={decorate}
        ]
        (0,0) -- (2,0);
\draw[xshift=2cm] (0,0)--(60:1) node[above] at (100:0.3){$-\e_3$};
\draw [xshift=2cm](0,0)--(-60:1) node[below] at (-80:0.3){$\e_4$};
\end{tikzpicture}
\end{center}
\caption{Spin graph evolution}
\label{flipgraphint}
\end{figure}
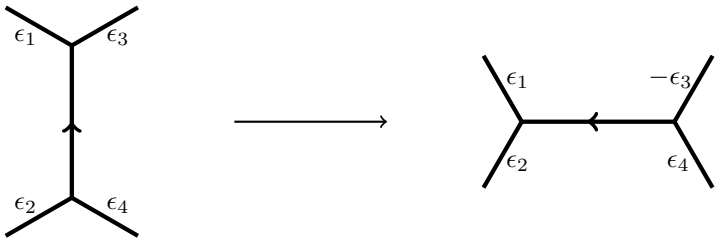
where $\epsilon_i$ stand for orientations on edges and extra minus signs stand for orientation reversal.

Following the original approach of Penner \cite{P2} and combining it with the description of spin structures, we obtain the Main Theorem stated in \cite{PZ}.

\begin{Thm} 
{\rm i)} The components of $\til{T}_{OSp(1|2)}(F)$ are determined by the space of spin structures on $F$. For each component $C$ of $\til{T}_{OSp(1|2)}(F)$, there are global affine coordinates on $C$ given by assigning to a triangulation $\Delta$ of $F$, 
\begin{itemize}
\item one even coordinate called the $\l$-length for each edge; 
\item one odd coordinate called the $\mu$-invariant for each triangle, taken modulo an overall change of sign.
\end{itemize} In particular we have a real-analytic homeomorphism: 
${C\to \mathbb{R}_{>0}^{6g-6+3s|4g-4+2s}/\mathbb{Z}_2.}$\\

{\rm ii)} The super Ptolemy transformations \cite{PZ} provide the analytic relations between coordinates assigned to different choice of triangulation $\Delta'$ of $F$, namely upon flip transformation. Explicitly (see Figure \ref{ptolemyint}),  when all $a,b,c,d$ are different edges of the triangulations of $F$, the Ptolemy transformations are as follows:

\begin{figure}[h!]

\centering

\begin{tikzpicture}[scale=0.5, baseline,ultra thick]

\draw (0,0)--(3,0)--(60:3)--cycle;

\draw (0,0)--(3,0)--(-60:3)--cycle;

\draw node[above] at (70:1.5){$a$};

\draw node[above] at (30:2.8){$b$};

\draw node[below] at (-30:2.8){$c$};

\draw node[below=-0.1] at (-70:1.5){$d$};

\draw node[above] at (1.5,-0.15){$e$};

\draw node[left] at (0,0) {};

\draw node[above] at (60:3) {};

\draw node[right] at (3,0) {};

\draw node[below] at (-60:3) {};

\draw node at (1.5,1){$\theta$};

\draw node at (1.5,-1){$\sigma$};

\end{tikzpicture}
\begin{tikzpicture}[baseline]

\draw[->, thick](0,0)--(1,0);

\node[above]  at (0.5,0) {};

\end{tikzpicture}
\begin{tikzpicture}[scale=0.5, baseline,ultra thick]

\draw (0,0)--(60:3)--(-60:3)--cycle;

\draw (3,0)--(60:3)--(-60:3)--cycle;

\draw node[above] at (70:1.5){$a$};

\draw node[above] at (30:2.8){$b$};

\draw node[below] at (-30:2.8){$c$};

\draw node[below=-0.1] at (-70:1.5){$d$};

\draw node[left] at (1.7,1){$f$};

\draw node[left] at (0,0) {};

\draw node[above] at (60:3) {};

\draw node[right] at (3,0) {};

\draw node[below] at (-60:3) {};

\draw node at (0.8,0){$\mu$};

\draw node at (2.2,0){$\nu$};

\end{tikzpicture}
\caption{$N=1$ Ptolemy transformation}
\label{ptolemyint}
\end{figure}
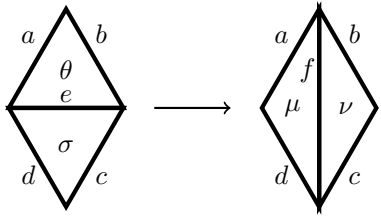
$$
ef=(ac+bd)\Big(1+\frac{\sigma\theta\sqrt{\mathcal{X}}}{1+\mathcal{X}}\Big),\quad \nu=\frac{\sigma+\theta\sqrt{\mathcal{X}}}{\sqrt{1+\mathcal{X}}},\quad
\mu=\frac{\sigma\sqrt{\mathcal{X}}-\theta}{\sqrt{1+\mathcal{X}}},
$$
where $\mathcal{X}=\frac{ac}{bd}$,  so that the evolution of arrows is as in  Figure \ref{flipgraphint}.
\end{Thm}

The whole idea of proof is based on the action of $OSp(1|2)$ in the super-Minkowski space 
$\mathbb{R}^{2,1|2}$ and the explicit construction of the images of ${\rm Hom}'$ of the group elements of $\pi_1(F)$ which depends on the spin structure. While in the paper \cite{PZ} one had to use  necessarily a bipartite fatgraph and a specific set of spin structures on it, the improved construction in this paper allows us to abolish the use of bipartite graph and start from any configuration.  

One can remove the decoration, passing from $\til{T}_{OSp(1|2)}(F)$ to ${T}_{OSp(1|2)}(F)$
 in the following way. To every edge $e$ (see Figure \ref{ptolemyint}) we can associate the {\it shear coordinate} $z_e:=\log (\frac{ac}{bd})$. These parameters satisfy a linear relation for every puncture, and together with the odd variables they form a set of coordinates on $T(F)$, thus allowing the removal of the decoration i.e., collapsing the $\mathbb{R}_+^s$-fiber, as it was in the pure even case. Here we mention that there is a physically and algebro-geometrically interesting refinement of $T(F)$ studied in \cite{IPZ2}, corresponding to the removal of certain odd degrees of freedom associated with R-punctures. 

\subsection{Outline of the paper}\label{subsec:sub}
Section \ref{sec:lightcone} is devoted to the description of superspace $\mathbb{R}^{2,2|4}$ modeled as the space of adjoint representations of the supergroup $\widetilde{SL}(1|2)\subset OSp(2|2)$. We fix notations and describe explicitly the action of certain important generators of $\widetilde{SL}(1|2)$. 

Again, as in \cite{PZ}, we take the orbit of the highest weight vector in adjoint representation as the light cone $\mathcal{L}_0$. We derive necessary functional relations on the components of the light-cone vectors. This is the subject of Section \ref{sec:orbits}.

In Section \ref{sec:orbitspair} we describe orbits of pairs and triples of vectors in $\mathcal{L}_0$. For the pairs such orbits are classified by a single even invariant $\lambda$-length, corresponding to the pairing of those two vectors. It turns out that the orbit of triple of linearly independent vectors in $\mathcal{L}_0$ is described by means of three $\lambda$-lengths and two odd parameters $(\theta_1,\theta_2)$ modulo permutation and rescaling $(\theta_1,\theta_2)\to (a\theta_1,a^{-1}\theta_2)$, where $a$ is any invertible even element. 

The invariants of a quadruple of points include one more even parameter in addition to 2 pairs of odd invariants and 5 $\lambda$-lengths. This extra invariant can be heuristically interpreted as a ``ratio" of the odd data for the two triples of points combined into the quadruple. In Section \ref{sec:basic} we describe in detail the orbits of 4 points, its invariants and related ambiguities for putting these 4 points in certain standard position.

In Section \ref{sec:spin} we return to the geometry of triangulated Riemann surfaces, recall and reformulate necessary facts about spin structures and connections on fatgraphs. In particular, we recall the description of spin structures given in \cite{PZ} based on equivalence classes of orientations on fatgraphs. 

Section \ref{sec:teich} is devoted to the main subject of the paper: description of coordinates on the $\cN=2$ super-Teichm\"uller space. Let $F:=F_g^s$ be a Riemann surface with genus $g\geq 0$ and $s\geq 1$ punctures such that $2g+s-2>0$. As usual, the construction of such coordinates involves lifting certain data assigned to the triangulation $\Delta$, or equivalently to the fatgraph $\tau$ dual to $\Delta$, to the light cone $\mathcal{L}_0$.  
The data giving the coordinate system $\til{C}(F,\D)$ are as follows:
\begin{itemize}\item we assign to each edge of $\D$ a positive even coordinate $e$;\\
\item we assign to each triangle of $\D$ two odd coordinates $(\h_1,\h_2)$;\\
\item for each triangle of $\D$, we assign to each of its edge $e$ a positive even coordinate $h_e$, called the \emph{ratio}, such that if $h_e$ and $h_e'$ are assigned to two triangles sharing the same edge $e$, then we have $h_e h_e'=1$.
\end{itemize}
The odd coordinates are defined up to an overall sign changes $\h_i\corr -\h_i$, as well as an overall involution $(\h_1,\h_2)\corr (\h_2,\h_1)$.

In particular, this assignment implies that the ratios $\{h_e\}$ uniquely define an $\R_+$-graph connection on $\t$. One can define the following action of the gauge transformations on the set of coordinates:  
if $h_a,h_b,h_e$ are ratios assigned to a triangle $T$ with odd coordinate $(\h_1,\h_2)$, then a \emph{vertex rescaling at $T$} is the following transformation:
\Eq{
(h_a,h_b,h_e,\h_1,\h_2)\mapsto (u h_a,u h_b,u h_e,u\inv\h_1,u\h_2)
} for some $u>0$, and all other coordinates fixed. 
We say that two coordinate vectors of $\til{C}(F,\D)$ are equivalent if they are related by a finite number of such vertex rescalings (i.e., gauge transformations). In particular, the underlying $\R_+$-graph connections on $\t$ are equivalent.

 Let $C(F,\D):=\til{C}(F,\D)/\sim$ be the equivalence classes of coordinate vectors. Then it can be represented by coordinates with $h_a h_b h_e=1$ for the ratios of a common triangle. This implies that
\Eq{
C(F,\D)\simeq \R_+^{8g+4s-7|8g+4s-8}/\Z_2\x\Z_2
} in accordance with the dimension for $\cN=2$ super-Teichm\"uller space given in \cite{Nat} (with an extra $s$-dimension given by decorations).

We can state the first main result of the paper:

\begin{Thm} Fix $F,\D,\t$ as before. Let 
$\w_\sigma$ and $\w_\iota$ be orientations
on $\tau$ representing respective spin structures $s_\sigma$ and $s_\iota$
on $F$.
%$\w_{sign}:=\w_{s_{sign},\t}$ be a representative,  corresponding to a specified spin structure $s_{sign}$ of $F$, and let $\w_{inv}:=\w_{s_{inv},\t}$ be the representative of another spin structure $s_{inv}$. 
Given a coordinate vector
$\vec[c]\in \til{C}(F,\D)$, there exists a map called the lift,
\Eqn{
\ell=\ell_{\w_{\sigma},\w_{\iota}}: \til{\D}_\oo \to \cL_0,}
from the vertices $\widetilde\Delta^\infty$ at infinity of a lift of $\Delta$ to the universal cover of $F$
to an appropriate set ${\mathcal L}_0$ of isotropic vectors in Minkowski space,
where $\ell$  is uniquely determined up to post-composition by $\SL(1|2)$ under certain admissibility conditions and only depends on the equivalence class $C(F,\D)$ of the coordinates. Moreover, there is a representation $\what{\rho}:\pi_1:=\pi_1(F)\to \SL(1|2)$, uniquely determined up to conjugacy by an element of $\SL(1|2)$, such that
\begin{itemize}\item[(1)] $\ell$ is $\pi_1$-equivariant, i.e., $\what{\rho}(\c)(\ell(a))=\ell(\c(a))$ for each $\c\in\pi_1$ and $a\in \til{\D}_\oo$;\\
\item[(2)] $\what{\rho}$ is a super Fuchsian representation, i.e., the natural projection $$\rho:\pi_1\xto{\what{\rho}} \SL(1|2) \to SL(2,\R)\to PSL(2,\R)$$ is a Fuchsian representation;\\
\item[(3)] the lift $\til{\rho}:\pi_1\xto{\what{\rho}} \SL(1|2) \to SL(2,\R)$ of $\rho$ does not depend on $\w_{\iota}$, and the space of all such lifts is in one-to-one correspondence with the spin structures $[\w_{\sigma}]\in\cO(\t)$ as equivalence classes of orientations on $\t$ as in section \ref{sec:spin}.
\end{itemize}
\end{Thm}
We call the space of all such lifts as in the theorem above the {\it decorated $\cN=2$ 
super-Teichm\"uller} space $S\til{T}(F)$ to distinguish it from the decorated $\cN=1$ super-Teichm\"uller space studied in \cite{PZ}.

In the construction of the lift $\ell$, which depends on two spin structures, we note that one spin structure controls the sign change for the odd invariants of attached triangles and the other one controls the change of the order of fermions. 
It is also important to mention that proving a similar result in \cite{PZ} we relied on an ad hoc construction involving bipartite graphs. The current result  
can be easily reduced to $\cN=1$ case, by considering lifts with equal value of odd coordinates per triangle and therefore abolishing the dependence on the second spin structure. The proof in this paper follows the general
constructions of \cite{PZ} now in this case $\cN=2$ but is
different conceptually from the one in \cite{PZ}: we no longer require the bipartite fatgraph as the starting point of the construction. Instead we use any fatgraph and given classes of orientations in order to construct the lift, so that $\cN=2$ is in fact more
closely analogous to the bosonic case $\cN=0$ in \cite{P1} than the treatment for $\cN=1$ in \cite{PZ}.

The next theorem, which follows from the one above, describes coordinates on $S\til{T}(F)$.

\begin{Thm} The components of $S\til{T}(F)$ are determined by two spin structures $s_{\sigma},s_{\iota}\in \cO(\t)$. For fixed representatives of the spin structures, $C(F,\D)$ provides global analytic coordinates on each component of $\R_+^{8g+4s-7|8g+4s-8}/\Z_2\x\Z_2$. 
\end{Thm}

In Section \ref{sec:ptolemy} we describe our second main result.  We  show explicitly how the coordinates behave under triangulation change and, in other words, give the description of generalized Ptolemy transformations. 

\begin{Thm}
In the generic situation where all $a,b,c,d$ are different edges of the triangulations of $F$, the Ptolemy transformations, corresponding to Figure \ref{ptolemylabel} below
\begin{figure}[h!]
\centering
\begin{tikzpicture}[scale=0.7, baseline,ultra thick]
\draw (0,0)--(3,0)--(60:3)--cycle;
\draw (0,0)--(3,0)--(-60:3)--cycle;
\draw node[above] at (70:1.5){$a$};
\draw node[above] at (30:2.8){$b$};
\draw node[below] at (-30:2.8){$c$};
\draw node[below=-0.1] at (-70:1.5){$d$};
\draw node[above] at (1.5,0){$e$};
\draw node[left] at (0,0) {};
\draw node[above] at (60:3) {};
\draw node[right] at (3,0) {};
\draw node[below] at (-60:3) {};
\draw node at (1.5,1){$\h_1,\h_2$};
\draw node at (1.5,-1){$\s_1,\s_2$};
\end{tikzpicture}
\begin{tikzpicture}[baseline]
\draw[->, thick](0,0)--(1,0);
\node[above]  at (0.5,0) {};
\end{tikzpicture}
\begin{tikzpicture}[scale=0.7, baseline,ultra thick]
\draw (0,0)--(60:3)--(-60:3)--cycle;
\draw (3,0)--(60:3)--(-60:3)--cycle;
\draw node[above] at (70:1.5){$a$};
\draw node[above] at (30:2.8){$b$};
\draw node[below] at (-30:2.8){$c$};
\draw node[below=-0.1] at (-70:1.5){$d$};
\draw node[left] at (1.5,1){$f$};
\draw node[left] at (0,0) {};
\draw node[above] at (60:3) {};
\draw node[right] at (3,0) {};
\draw node[below] at (-60:3) {};
\draw node at (0.8,0){$\mu_1,\mu_2$};
\draw node at (2.2,0){$\nu_1,\nu_2$};
\end{tikzpicture}\\
\begin{tikzpicture}[scale=0.8, baseline,ultra thick]
\draw (0,0)--(3,0)--(60:3)--cycle;
\draw (0,0)--(3,0)--(-60:3)--cycle;
\draw node[right] (ha) at (60:1.2){$h_a$};
\draw node[right =0.2 of ha] {$h_b$};
\draw node[right] (hd) at (-60:1.2){$h_d$};
\draw node[right =0.2 of hd] {$h_c$};
\draw node[above] at (1.5,0){$h_e$};
\draw node[below] at (1.5,0){$h_e\inv$};
\draw node[left] at (0,0) {};
\draw node[above] at (60:3) {};
\draw node[right] at (3,0) {};
\draw node[below] at (-60:3) {};
\end{tikzpicture}
\begin{tikzpicture}[baseline]
\draw[->, thick](0,0)--(1,0);
\node[above]  at (0.5,0) {};
\end{tikzpicture}
\begin{tikzpicture}[scale=0.8, baseline,ultra thick]
\draw (0,0)--(60:3)--(-60:3)--cycle;
\draw (3,0)--(60:3)--(-60:3)--cycle;
\draw node[right] (ha) at (60:1.2){$h_a'$};
\draw node[right =0.2 of ha] {$h_b'$};
\draw node[right] (hd) at (-60:1.2){$h_d'$};
\draw node[right =0.2 of hd] {$h_c'$};
\draw node[left] at (1.5,0){$h_f$};
\draw node[right] at (1.5,0){$h_f\inv$};
\draw node[left] at (0,0) {};
\draw node[above] at (60:3) {};
\draw node[right] at (3,0) {};
\draw node[below] at (-60:3) {};
\end{tikzpicture}\\
\caption{$N=2$ Ptolemy transformation}\label{ptolemylabel}
\end{figure}
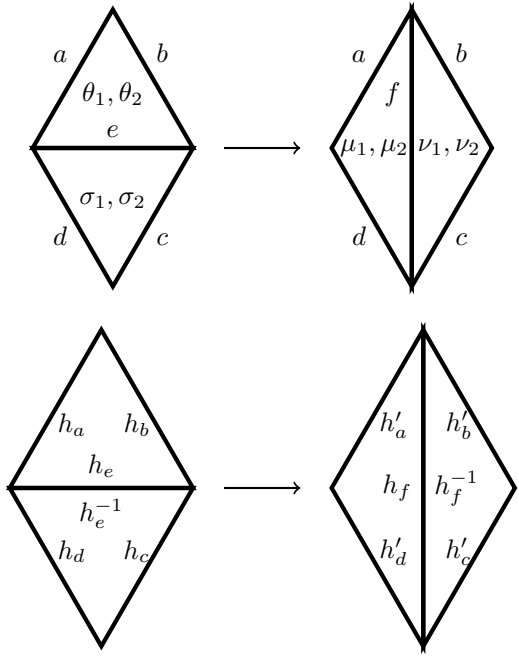
  %%%%%%%%  %%%%%%%%  %%%%%%%%
are as follows:
\Eq{
ef&=(ac+bd)\left(1+\frac{h_e\inv\s_1\h_2}{2(\sqrt\mathcal{X}+\sqrt\mathcal{X}\inv)}+\frac{h_e\s_2\h_1}{2(\sqrt\mathcal{X}+\sqrt\mathcal{X}\inv)}\right),
}
\Eq{
\mu_1=\frac{h_e\h_1+\sqrt\mathcal{X} \s_1}{\cD},\tab\mu_2=\frac{h_e\inv\h_2+\sqrt\mathcal{X} \s_2}{\cD}, \\
\nu_1=\frac{\s_1-\sqrt\mathcal{X} h_e\h_1}{\cD},\tab \nu_2=\frac{\s_2-\sqrt\mathcal{X} h_e\inv\h_2}{\cD},
}
\Eq{
h_a' = \frac{h_a}{h_e c_\h},\tab h_b' = \frac{h_b c_\h}{h_e},\tab h_c' =h_c\frac{c_\h}{c_\mu}, \tab h_d'=h_d\frac{c_\nu}{c_\h},\tab h_f =\frac{c_\s}{c_\h^2},
}
where \Eq{
\cD:=\sqrt{1+\mathcal{X}+\frac{\sqrt\mathcal{X}}{2}(h_e\inv\s_1\h_2+h_e\s_2\h_1)}
} and $c_{\theta}:=1+\frac{\theta_1\theta_2}{6}$, while the signs of the fermions follow from the construction associated to the spin graph evolution rule in Figure \ref{flipgraphint}, where $\e_i$ denote the orientations of the edges.
\end{Thm}
In Appendix \ref{sec:specialPtolemy}, we also write down formulas for flip transformation in the other cases where some edges are identified. In particular we write explicit formulas in the case of the once-punctured torus and the thrice-punctured sphere.

One prominent structure from \cite{PZ} which we did not generalize to the $\cN=2$ case is the analogue of the Weil-Petersson K\"ahler two-form.  We leave this and other questions for subsequent publications.

\textbf{Acknowledgment} The first author is supported by JSPS KAKENHI Grant Numbers JP26800004, JP16K17571 and Top Global University Project, MEXT, Japan. 
%==============================================================================
\section{Light cone and $\SL(1|2)$-action}\label{sec:lightcone}
Let $SL(1|2)$ be the $(2|1)\x (2|1)$ supermatrices with superdeterminant equal to 1, and $\sl(1|2)$ the corresponding Lie superalgebra described in Appendix \ref{sec:SL12}. Also let $\SL(1|2)$ be the semidirect product $\Psi\ltimes SL(1|2)_0$ of the involution $\Psi$ and the component $SL(1|2)_0$ described in Definition \ref{SL0}. (See also Remark \ref{SL12choice} for the explanation of the choice of $SL(1|2)$ over the more conventional $OSp(2|2)$.) Consider the adjoint action of $SL(1|2)$ acting on the pure even elements of the form
\Eqn{
M:&=x_1E-x_2F-yH+\xi_1^+e_1^++\xi_2^+e_2^++\xi_1^-e_1^-+\xi_2^-e_2^-+z(h_1+h_2)\\
&=\veca{z-y&\xi_1^+&x_1\\\xi_1^-&2z&\xi_2^+\\-x_2&\xi_2^-&z+y}\\
&:=(x_1,x_2,y,z|\xi_1^+,\xi_2^+,\xi_1^-,\xi_2^-)\in\R^{2,2|4},
}
where the invariant quadratic form providing the Minkowski space structure $\R^{2,2|4}$ is given by the supertrace:
\Eq{
\<M,M\>&:=-\frac{1}{2} str(M^2)\nonumber\\
&=-\frac12 (-2x_1x_2+2y^2-2\xi_1^+\xi_1^--2\xi_2^+\xi_2^--2z^2)\nonumber\\
&=x_1x_2-y^2+\xi_1^+\xi_1^--\xi_2^+\xi_2^-+z^2.
}
\begin{Def}
The light cone $\cL\subset \R^{2,2|4}$ is defined to be the set 
\Eq{
\cL=\{(x_1,x_2,y,z|\xi_1^+,\xi_2^+,\xi_1^-,\xi_2^-)\in\R^{2,2|4}: x_1x_2-y^2+\xi_1^+\xi_1^--\xi_2^+\xi_2^-+z^2=0\},
}
i.e., the points $M\in \R^{2,2|4}$ such that $\<M,M\>=0$.
\end{Def}

In particular, we have an inner product
\Eq{
\label{pairing}
\<M,M'\> = \frac{1}{2}(x_1x_2'+x_2x_1')-yy'+\frac{1}{2}(\xi_1^+{\xi_1^-}'-\xi_2^+{\xi_2^-}'-\xi_1^-{\xi_1^+}'+\xi_2^-{\xi_2^+}')+zz',
}
and we shall refer to the square root of $\<M,M'\>$ as the \emph{$\l$-length} between the points $M$ and $M'$.

\begin{Def}\label{notation}
We denote certain elements of $SL(1|2)$ as follows:
\Eqn{
D_{a,c}&:=\veca{a&0&0\\0&ac&0\\0&0&c},\tab D_a:=D_{a,a\inv},\tab Z_a:=D_{a,a},\tab J:=\veca{0&0&1\\0&1&0\\-1&0&0},\\
U_\a &:=\veca{1&\a&0\\0&1&0\\0&0&1},\tab V_\b:=\veca{1&0&0\\0&1&\b\\0&0&0},\tab W_b=\veca{1&0&b\\0&1&0\\0&0&1}.
}
\end{Def}
\begin{Lem}\label{action} The adjoint actions on $\cL$ by conjugation are given by
\Eqn{
D_{a,c}\cdot (x_1,x_2,y,z|\xi_1^+,\xi_2^+,\xi_1^-,\xi_2^-)&=(ac\inv x_1,a\inv cx_2,y,z|c\inv \xi_1^+,a\xi_2^+,c\xi_1^-,a\inv \xi_2^-),\\
D_a\cdot (x_1,x_2,y,z|\xi_1^+,\xi_2^+,\xi_1^-,\xi_2^-)&=(a^2 x_1,a^{-2} x_2,y,z|a \xi_1^+,a\xi_2^+,a\inv\xi_1^-,a\inv \xi_2^-),\\
Z_a\cdot (x_1,x_2,y,z|\xi_1^+,\xi_2^+,\xi_1^-,\xi_2^-)&= (x_1,x_2,y,z|a\inv \xi_1^+,a\xi_2^+,a\xi_1^-,a\inv \xi_2^-),\\
J\cdot (x_1,x_2,y,z|\xi_1^+,\xi_2^+,\xi_1^-,\xi_2^-)&= (x_2,x_1,-y,z| \xi_2^-,-\xi_1^-,\xi_2^+,-\xi_1^+),
}
\Eqn{
&U_\a\cdot (x_1,x_2,y,z|\xi_1^+,\xi_2^+,\xi_1^-,\xi_2^-)\\
&\tab=(x_1-\a\xi_2^+,x_2,y+\frac{\a\xi_1^-}{2},z-\frac{\a\xi_1^-}{2}|\xi_1^++(y+z)\a,\xi_2^+,\xi_1^-,\xi_2^-+x_2\a),\\
&V_\b\cdot (x_1,x_2,y,z|\xi_1^+,\xi_2^+,\xi_1^-,\xi_2^-)\\
&\tab=(x_1-\b\xi_1^+,x_2,y-\frac{\b\xi_2^-}{2},z-\frac{\b\xi_2^-}{2}|\xi_1^+,\xi_2^++(y-z)\b,\xi_1^--x_2\b,\xi_2^-),\\
&W_b\cdot (x_1,x_2,y,z|\xi_1^+,\xi_2^+,\xi_1^-,\xi_2^-)\\
&\tab=(x_1+b^2x_2+2by,x_2,y+bx_2,z|\xi_1^++b\xi_2^-,-b\xi_1^-+\xi_2^+,\xi_1^-,\xi_2^-).
}
\end{Lem}
\begin{proof}
The actions of bosonic elements follows from the usual matrix calculation. Let us explicate the action for $U_\a$. The case for $V_\b$ is similar. We have
$U_\a\inv = \veca{1&-\a&0\\0&1&0\\0&0&1}$, hence following the convention in \eqref{signs},
\Eqn{
&U_\a\cdot (x_1,x_2,y,z|\xi_1^+,\xi_2^+,\xi_1^-,\xi_2^-)\\
 &=\veca{1&\a&0\\0&1&0\\0&0&1}\veca{z-y&\xi_1^+&x_1\\\xi_1^-&2z&\xi_2^+\\-x_2&\xi_2^-&z+y}\veca{1&-\a&0\\0&1&0\\0&0&1}\\
 &=\veca{z-y-\a\xi_1^-&\xi_1^++2z\a&x_1-\a\xi_2^+\\\xi_1^-&2z&\xi_2^+\\-x_2&\xi_2^-&z+y}\veca{1&-\a&0\\0&1&0\\0&0&1}\\
 &=\veca{z-y-\a\xi_1^-&\xi_1^++(y+z)\a&x_1-\a\xi_2^+\\\xi_1^-&2z+\xi_1^-\a&\xi_2^+\\-x_2&\xi_2^-+x_2\a&z+y}\\
 &=(x_1-\a\xi_2^+,x_2,y+\frac{\a\xi_1^-}{2},z-\frac{\a\xi_1^-}{2}|\xi_1^++(y+z)\a,\xi_2^+,\xi_1^-,\xi_2^-+x_2\a).
 }
\end{proof}
Let us also consider the action of the  involution $\Psi$ defined in Proposition \ref{involution}.
\begin{Lem} The involution $\Psi$ acts on $\cL$ by
\Eqn{
\Psi\cdot (x_1,x_2,y,z|\xi_1^+,\xi_2^+,\xi_1^-,\xi_2^-)=(x_1,x_2,y,-z|\xi_2^+,\xi_1^+,-\xi_2^-,-\xi_1^-).
}
\end{Lem}
The lower Borel elements and their inverses are
\Eqn{
B=\veca{a&0&0\\\a&f&0\\c&\b&d},\tab B\inv = \veca{a\inv&0&0\\-\frac{\a}{af}&f\inv&0\\-\frac{c}{f}+\frac{\a\b}{f^2}&-\frac{\b}{fd}&d\inv},
}
where by definition of the superdeterminant we have $f=ad$. Finally we also embed the $SL(2)$ subgroup as
\Eqn{
g=\veca{a&0&b\\0&1&0\\c&0&d}\in SL(2)\inj SL(1|2),
}
where $ad-bc=1$.
%==============================================================================
\section{Orbit of $\SL(1|2)$ in the light cone}\label{sec:orbits}
In this section we would like to describe a particular orbit of $\SL(1|2)$, namely the orbit $\cL_0$ of $$e_0:=E=(1,0,0,0|0,0,0,0)\in\cL$$ under the adjoint action, which allows us to parametrize all representations of $\pi_1(F)$. Following \cite{PZ}, we shall call $\cL_0$ the \emph{special light cone}.
\begin{Rem}
In the $\cN=1$ case, the orbit $\mathcal{L}_0$ of the highest weight vector in the adjoint representation equivariantly projects onto the $\cN=1$ extension of the boundary of the upper half-plane, so that the $OSp(1|2)$ action is realized via fractional-linear superconformal transformations on $\mathbb{R}^{1|1}$. This gives $\mathcal{L}_0$ a meaning of the $\cN=1$ version of the original light cone.
 
In the $\cN=2$ case one can perform the same calculation and show that the orbit $\mathcal{L}_0$ has the same meaning, namely there is an equivariant projection onto the $\cN=2$ extension of the boundary of $\cN=2$ upper half-plane, i.e. $\mathbb{R}^{1|2}$. We leave the details for the interested reader.

\end{Rem}
\begin{Prop}\label{orbit_relation}
The elements $(x_1,x_2,y,z|\xi_1^+,\xi_2^+,\xi_1^-,\xi_2^-)$ of the light cone $\cL$ such that $x_1,x_2$ have non-negative bodies and satisfying
\Eq{\label{*}
\case{
\xi_1^-=-\frac{y}{x_1} \xi_2^+,\tab\xi_2^-=\frac{y}{x_1}\xi_1^+,\tab z=\frac{\xi_1^+\xi_2^+}{2x_1} & \mbox{if $x_1\neq 0$},\\
\xi_2^+=-\frac{y}{x_2}\xi_1^-,\tab\xi_1^+=\frac{y}{x_2}\xi_2^-,\tab z=\frac{\xi_1^-\xi_2^-}{2x_2} & \mbox{if $x_2\neq 0$},
}
}
constitute the orbit $\cL_0$ of $e_0=(1,0,0,0|0,0,0,0)$.
\end{Prop}
\begin{proof}
Note that if an element $M\in\cL$ satisfies \eqref{*}, then for $x_1\neq 0$ we have
\Eqn{
0&=x_1x_2-y^2+\xi_1^+\xi_1^--\xi_2^+\xi_2^-+z^2\\
&=x_1x_2-y^2-\frac{y}{x_1}\xi_1^+\xi_2^+-\frac{y}{x_1}\xi_2^+\xi_1^++\frac{(\xi_1^+\xi_2^+)^2}{4x_1^2}\\
&=x_1x_2-y^2,
}
and similarly for the case of $x_2\neq 0$. Hence we have $y^2=x_1x_2$ for these elements.

Now consider $M=(x_1,x_2,y,z|\xi_1^+,\xi_2^+,\xi_1^-,\xi_2^-)$ that satisfies \eqref{*}. Acting on it by an element $g\in SL(2)$, we can assume that $x_1\neq 0, x_2\neq 0$. Then acting on it by the diagonal matrix $D_{a,c}$ with $a=\sqrt{\frac{x_2}{x_1}},c=1$ we can assume $x_1=x_2=t$. Since $y^2=x_1x_2=t^2$ we have $y=\pm t$. By the action of $J$ we can further assume $y=t$.

Now consider the action of the lower Borel acting on $e_0$:
\Eq{
B\cdot e_0 &=\veca{a&0&0\\\a&f&0\\c&\b&d}\veca{0&0&1\\0&0&0\\0&0&0}\veca{a\inv&0&0\\-\frac{\a}{af}&f\inv&0\\-\frac{c}{f}+\frac{\a\b}{f^2}&-\frac{\b}{fd}&d\inv}\nonumber\\
&=\veca{-\frac{c}{d}+\frac{\a\b}{fd}&-\frac{a\b}{fd}&\frac{a}{d}\\-\frac{\a c}{f}&\frac{\a\b}{fd}&\frac{\a}{d}\\-\frac{c^2}{f}+\frac{c\a\b}{f^2}&-\frac{c\b}{fd}&\frac{c}{d}}.\label{lowerB}
}
Hence solving for $B\cdot e_0=M$ for $y=t$, one of the solutions is given by
\Eqn{
a=t\tab d=1,\tab f=t,\tab c=\frac{\xi_1^+\xi_2^+}{2t}+ t,\tab \a=\xi_2^+,\tab \b=-\xi_1^+.
}
Hence an element satisfying \eqref{*} lies in the orbit of $e_0$.

Next, we note that for any connected Lie (super)group, a neighborhood of the identity generates the whole group, therefore, to show that the orbit of $e_0$ satisfies \eqref{*} it suffices to show that the relations \eqref{*} are preserved under the action of the infinitesimal group transformations and the involution $\Psi$ (since $\Psi$ gives a second connected component of $\widetilde{SL}(1|2)$). From \eqref{lowerB} we see easily that the relations \eqref{*} are preserved under the lower Borel action. Also by Lemma \ref{action} the relations are clearly preserved by the involutions $J$ and $\Psi$. Hence it suffices to consider the action of $U_\a,V_\b$. By the action of $J$ which interchanges $x_1$ and $x_2$, it suffices to consider the case for $x_2\neq 0$.

By Lemma \ref{action} the action of $U_\a$ is given by
\Eqn{
&U_\a\cdot (x_1,x_2,y,z|\xi_1^+,\xi_2^+,\xi_1^-,\xi_2^-)\\
&=(x_1-\a\xi_2^+,x_2,y+\frac{\a\xi_1^-}{2},z-\frac{\a\xi_1^-}{2}|\xi_1^++(y+z)\a,\xi_2^+,\xi_1^-,\xi_2^-+x_2\a)\\
&=:(x_1',x_2',y',z'|{\xi_1^+}',{\xi_2^+}',{\xi_1^-}',{\xi_2^-}').
}
Then 
\Eqn{
-\frac{y'}{x_2'}{\xi_1^-}'&=(-\frac{y}{x_2}-\frac{\a\xi_1^-}{2x_2})\xi_1^-=-\frac{y}{x_2}\xi_1^-=\xi_2^+={\xi_2^+}',\\
\frac{y'}{x_2'}{\xi_2^-}'&=(\frac{y}{x_2}+\frac{\a\xi_1^-}{2x_2})(\xi_2^-+x_2\a)=\frac{y}{x_2}\xi_2^-+y\a+\frac{\a\xi_1^-\xi_2^-}{2x_2}=\xi_1^++y\a+z\a={\xi_1^+}',\\
\frac{{\xi_1^-}'{\xi_2^-}'}{2x_2'}&=\frac{\xi_1^-(\xi_2^-+x_2\a)}{2x_2}=\frac{\xi_1^-\xi_2^-}{2x_2}-\frac{\a\xi_1^-}{2}=z'.
}
Similarly for the action of $V_\b$ we have
\Eqn{
&V_\b\cdot (x_1,x_2,y,z|\xi_1^+,\xi_2^+,\xi_1^-,\xi_2^-)\\
&=(x_1-\b\xi_1^+,x_2,y-\frac{\b\xi_2^-}{2},z-\frac{\b\xi_2^-}{2}|\xi_1^+,\xi_2^++(y-z)\b,\xi_1^--x_2\b,\xi_2^-)\\
&=:(x_1'',x_2'',y'',z''|{\xi_1^+}'',{\xi_2^+}'',{\xi_1^-}'',{\xi_2^-}'').
}
Then
\Eqn{
-\frac{y''}{x_2''}{\xi_1^-}''&=(\frac{y}{x_2}-\frac{\b\xi_2^-}{2x_2})(\xi_1^--x_2\b)=\frac{y\xi_1^-}{x_2}-y\b+\frac{\b\xi_1^-\xi_2^-}{2x_2}=-\xi_2^+-y\b+z\b=-{\xi_2^+}''\\
\frac{y''}{x_2''}{\xi_2^-}''&=(\frac{y}{x_2}-\frac{\b\xi_2^-}{2x_2})\xi_2^-=\frac{y}{x_2}\xi_2^-=\xi_1^+={\xi_1^+}''\\
\frac{{\xi_1^-}''{\xi_2^-}''}{2x_2''}&=\frac{(\xi_1^--x_2\b)\xi_2^-}{2x_2}=\frac{\xi_1^-\xi_2^-}{2x_2}-\frac{\b\xi_2^-}{2}=z''
}
and this completes the proof.
\end{proof}

%==============================================================================
\section{Orbits of pairs and triples in the light cone}\label{sec:orbitspair}
In this section we study the space of $\til{SL}(1|2)$-orbits of linearly independent ordered pairs and positively oriented triples of points in the special light cone $\cL_0$.

\begin{Prop}For two points on $\cL_0$ which are linearly independent (i.e., the pairing has non-zero body), the space of orbits is described by one parameter, which will be called the $\l$-length.
\end{Prop}
\begin{proof} Let $v_1, v_2$ be two vectors on $\cL_0$. By means of the action of $SL(1|2)$, we can bring $v_1=-F$, i.e., $x_2=1$ and all other components zero. Since we assume $\<v_1,v_2\>\neq 0$, the $x_1$-component of $v_2$ is non-zero. Therefore as in the previous section one can use an element $B$ of the lower Borel subgroup of $SL(1|2)$ in order to reduce it to the vector $v_2=E$, i.e., $x_1=1$ and all other components zero. At the same time, $B\cdot v_1=\l^2 v_1$ for some scalar $\l>0$ since $F$ is a lowest weight vector of the adjoint action. This scalar multiple is exactly the pairing $\<v_1,v_2\>$.
\end{proof}
\begin{Def} We call an ordered triple of points $(A,B,C):=\D ABC$ on $\cL_0$ \emph{positively ordered} if the triple of points are linearly independent, and the bodies of the underlying bosonic triples $(x_1,x_2,y)$ constitute a positively oriented basis of $\R^{2,1}$.
\end{Def}
\begin{Def} A positively ordered triple $\D ABC\subset \cL_0$ is said to be in \emph{standard position} with odd parameters $\h:=(\h_1,\h_2)$ if the points comprising the triple are of the form
\Eq{\label{standard}
A=&r(0,1,0,0|0,0,0,0),\\
B=&t(1,1,1,\frac{\h_1\h_2}{2}|\h_1,\h_2,-\h_2,\h_1),\\
C=&s(1,0,0,0|0,0,0,0),
}
for some positive numbers $r,s,t\in\R_{>0}$ and a pair of fermions $\h_1,\h_2$. We denote this by $\D ABC\in \cS^\h$.
\end{Def}
\begin{Prop}\label{stab}
$\D ABC\in \cS^\h$ and $\D A'B'C'\in \cS^{\h'}$ are in the same $SL(1|2)$-orbit if and only if $(\h_1',\h_2')=(a\h_1,a\inv \h_2)$ for some $a\neq 0$. In particular the stabilizers of $\D ABC\in \cS^\h$ are of the form $Z_a\in \cZ$ where $a\h_i=\h_i$ for $i=1,2$.
\end{Prop}
\begin{proof} It is easy to check that the stabilizers for both $A$ and $C$ are the $\cZ$-subgroup elements. Hence the result follows from the action of $\cZ$ given in Lemma \ref{action}.
\end{proof}

\begin{Prop}\label{orbit} The space of $\SL(1|2)$-orbits of positively ordered triples of points on $\cL_0$ is parametrized by 3 positive numbers $r,s,t\in\R_{>0}$ and an unordered pair of fermions $\{\h_1,\h_2\}$ modulo the transformation $\{\h_1,\h_2\}\sim \{a\h_1,a\inv \h_2\}$ where $a$ is invertible. We call the equivalence class $[\h]$ of $\{\h_1,\h_2\}$ the odd invariant of the ordered triple.
\end{Prop}
\begin{proof}
Consider the positively ordered triple $\D ABC$. Let us put $C$ into the position $C=s(1,0,0,0|0,0,0,0)$ by an $SL(1|2)$ transform. Using Lemma \ref{action}, one sees that $U_\a,V_\b$ and the $SL(2)$ element $W_b$ lie in the stabilizer of $C$. Now given $A=(x_1,x_2,y,z|\xi_1^+,\xi_2^+,\xi_1^-,\xi_2^-)$, we can assume $x_2\neq 0$ (since otherwise the even part will be parallel to $C$), we can apply $V_\b U_\a$ to $A$ where $\a=-\frac{\xi_2^-}{x_2}$ and $\b=\frac{\xi_1^-}{x_2}$ and get $A'=(x_1',x_2',y',0|0,0,0,0)$. Finally applying $W_b$ where $b=-\frac{y'}{x_2'}$ we bring $A$ into the form $r(0,1,0,0|0,0,0,0)$.

Now $B$ will be in the general form $(x_1,x_2,y,z|\xi_1^+,\xi_2^+,\xi_1^-,\xi_2^-)$ with $x_1\neq 0$. Applying the diagonal transformation $D_{a,a\inv}$ with $a=\sqrt[4]{\frac{x_2}{x_1}}$ we can bring $B$ to the form $t(1,1,y,z|\h_1,\h_2,\h_1',\h_2')$. Since $y^2=x_1x_2$ on the orbit, we have $y=1$. From the relations \eqref{*} we see that $z=\frac{\h_1\h_2}{2},\h_1'=-\h_2,\h_2'=\h_1$. Hence the ordered triple $\D ABC\in \cS^\h$ is brought into standard position.

By Lemma \ref{action}, $Z_a$ acts on the triple by mapping the odd parameters $(\h_1,\h_2)\mapsto (a\inv\h_1, a\h_2)$, hence the parameters are determined up to rescaling. Finally the involution acts by interchanging $\h_1$ and $\h_2$, hence the orbit does not depend on the order of the fermions.
\end{proof}
\begin{Prop} The odd invariant $[\h]$ does not depend on the order of the vertices in standard position.
\end{Prop}
\begin{proof}
we shall show that the  cyclic permutation of the vertices is not relevant for $\h_1,\h_2$ modulo rescaling. To do that we shall construct the analogue of the ``prime transformation" from \cite{PZ} in this case by rotating the vertices. Recall from \eqref{lowerB}, that we have
\Eqn{
\cU (1,0,0,0|0,0,0,0)\cU\inv = (1,1,1,\frac{\h_1\h_2}{2}|\h_1,\h_2,-\h_2,\h_1),
}
where
\Eqn{
\cU=\veca{1&0&0\\\h_2&1&0\\1+\frac{\h_1\h_2}{2}&-\h_1&1},\tab \cU\inv = \veca{1&0&0\\-\h_2&1&0\\\frac{\h_1\h_2}{2}-1&\h_1&1}.
}

Let $\D ABC\in \cS^\h$. Applying $\cU\inv$ to $\D ABC$ gives
\Eqn{
A&\mapsto r(0,1,0,0|0,0,0,0),\\
B&\mapsto t(1,0,0,0|0,0,0,0),\\
C&\mapsto s(1,1,1,\frac{\h_1\h_2}{2}|-\h_1,-\h_2,-\h_2,\h_1).
}
Finally apply $J$ from Lemma \ref{action} which maps
\Eqn{
A&\mapsto r(1,0,0,0|0,0,0,0),\\
B&\mapsto t(0,1,0,0|0,0,0,0),\\
C&\mapsto s(1,1,1,\frac{\h_1\h_2}{2}|\h_1,\h_2,-\h_2,\h_1).
}
Hence the ordered triple $\D BCA$ is now in standard form with the same odd parameters.
\end{proof}

The prime transformation $P'\in SL(1|2)$ is therefore given by
\Eqn{
P'&=J\circ \cU\inv\\
&=\veca{0&0&1\\0&1&0\\-1&0&0}\veca{1&0&0\\-\h_2&1&0\\\frac{\h_1\h_2}{2}-1&\h_1&1}\\
&=\veca{\frac{\h_1\h_2}{2}-1&\h_1&1\\-\h_2&1&0\\-1&0&0}.
}
However, one checks that actually
\Eqn{
{P'}^3=\veca{1-\frac{\h_1\h_2}{2}&0&0\\0&1-\h_1\h_2&0\\0&0&1-\frac{\h_1\h_2}{2}}=Z_{1-\frac{\h_1\h_2}{2}}\in \cZ.
}

First let us introduce the following notations:

\begin{Def}For notational convenience, we write $\h:=(\h_1,\h_2)$, $a\h:=(a\h_1,a\inv \h_2)$ for $a\neq 0$, and $\h^{op}:=(\h_2,\h_1)$. We shall write $a^{\pm_i}:=a^{3-2i}=\case{a&i=1\\a\inv&i=2}$.  

We shall denote the constant
\Eq{c_\h:=\frac{1+\h_1\h_2}{6}.}
Then $c_\h$ is invariant under the rescaling $\h\to a\h$. 

Finally we write $\h\in[\D ABC]$ if $\h$ is a representative of the odd invariant $[\h]$ of the ordered triple $\D ABC$. 
\end{Def}

Now we can define the following more general prime transformations. Let $h_A,h_B,h_C$ be some positive numbers. 

\begin{Def}
The general prime transformation $P_{h_B,h_C}^{\h}$, which transforms $\D ABC \in \cS^{h_B \h}$ to $\D BCA\in \cS^{h_C\h}$ is given by
\Eqn{
P_{h_B,h_C}^{\h}:= Z_{c_\h}\circ Z_{h_C}\inv\circ P'\circ Z_{h_B}.
}
\end{Def}

\begin{Prop} We have the relations
$$P_{h_A,h_B}^\h\circ P_{h_C,h_A}^\h\circ P_{h_B,h_C}^\h = 1.$$
\end{Prop}

By direct calculation we obtain
\begin{Lem}\label{primeaction}
The action of the general prime transformation $P_{h_B,h_C}^\h$ on $\cL_0$ is given by
\Eqn{
P_{h_B,h_C}^\h \cdot (x_1,x_2,y,z|\xi_1^+,\xi_2^+,\xi_1^-,\xi_2^-)=({x_1}',{x_2}',{y}',{z}'|{\xi_1^+}',{\xi_2^+}',{\xi_1^-}',{\xi_2^-}'),
}
where
\Eqn{
{x_2}'&:=x_1,\\
{z}'&:=z+\frac{1}{2}(h_B\inv\h_2\xi_1^+-h_B \h_1\xi_2^++\h_1\h_2 x_1),\\
{y}'&:=-y+\frac{1}{2}(h_B\inv\h_2\xi_1^++h_B\h_1\xi_2^+)+x_1,\\
{\xi_1^-}'&:=c_\h h_C\inv(h_B\xi_2^+-\h_2x_1),\\
{\xi_2^-}'&:=c_\h\inv h_C(-h_B\inv\xi_1^++\h_1x_1),
}
and by \eqref{*} we have
\Eqn{
{x_1}'=\frac{{y'}^2}{{x_2}'},\tab {\xi_1^+}'=\frac{{y}'}{{x_2}'}{\xi_2^-}',\tab {\xi_2^+}'=-\frac{{y}'}{{x_2}'}{\xi_1^-}'.
}
\end{Lem}

For notational convenience, we shall denote edges as well as their $\l$-lengths by the same (Latin lowercase) letter as long as it does not lead to confusion.  
The real positive numbers $h$, corresponding to the given vertex will from now on be labeled by the edge opposite to that vertex. 
We shall sometimes denote
\Eq{\label{leftright}
P_{h_a,h_b}^{\h,+}:=P_{h_a,h_b}^\h,\tab P_{h_a, h_b}^{\h,-}:= (P_{h_b,h_a}^{\h})\inv.
}
By definition,
\Eqn{
P_{h_a,h_b}^{\h,+}&=Z_{c_\h}\circ Z_{h_b}\inv \circ P'\circ  Z_{h_a},\\
P_{h_a,h_b}^{\h,-}&=Z_{c_\h}\inv \circ Z_{h_b}\inv \circ {P'}\inv \circ Z_{h_a},
}
which denote respectively the clockwise and anti-clockwise prime transformations.
%==============================================================================
\section{Basic calculation}\label{sec:basic}
In this section we deal with orbits of 4-tuples $(A,B,C,D):=\dia ABCD$ in the special light cone $\cL_0$ such that $\D ABC$ and $\D CDA$ are positively ordered triples of points. The derivation of the coordinates of the ``standard position'' of such quadrilateral, which was called the \emph{basic calculation} \cite{P2} in the pure even case, constitutes the heart of the definition of the super-Teichm\"uller space in Section \ref{sec:teich}.

Recall the pairing \eqref{pairing} that is invariant under the action of $\SL(1|2)$:
\Eq{
\<M,M'\> = \frac{1}{2}(x_1x_2'+x_2x_1')-yy'+\frac{1}{2}(\xi_1^+{\xi_1^-}'-\xi_2^+{\xi_2^-}'-\xi_1^-{\xi_1^+}'+\xi_2^-{\xi_2^+}')+zz'.
}

%%%%%%%% NICE PICTURE %%%%%%%%%

\begin{figure}[h!]
\centering
\begin{tikzpicture}[ultra thick]
\draw (0,0)--(3,0)--(60:3)--cycle;
\draw (0,0)--(3,0)--(-60:3)--cycle;
\draw node[above] at (70:1.5){$a$};
\draw node[above] at (30:2.8){$b$};
\draw node[below] at (-30:2.8){$c$};
\draw node[below=-0.1] at (-70:1.5){$d$};
\draw node[above] at (1.5,0){$e$};
\draw node[left] at (0,0) {$A$};
\draw node[above] at (60:3) {$B$};
\draw node[right] at (3,0) {$C$};
\draw node[below] at (-60:3) {$D$};
\draw node at (1.5,1){$\h_1,\h_2$};
\draw node at (1.5,-1){$\s_1,\s_2$};
\end{tikzpicture}
\caption{Labelling with $A,B,C$ in standard position}
\label{basic}
\end{figure}
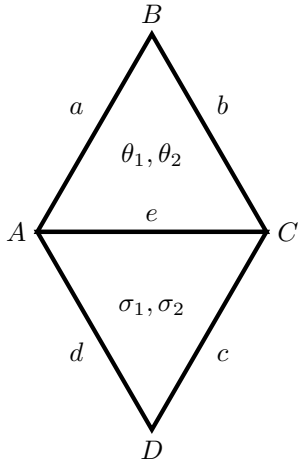

%%%%%%%%%%%%%%%%%%%%%%%%%

Let $\D ABC\in \cS^\h$ be in standard position
\Eqn{
A&=r(0,1,0,0|0,0,0,0),\\
B&=t(1,1,1,\frac{\h_1\h_2}{2}|\h_1,\h_2,-\h_2,\h_1),\\
C&=s(1,0,0,0|0,0,0,0).
}
Then (the squares of) $\l$-lengths of the edges by definition are given by
\Eq{
a^2=\<A,B\>=\frac{rt}{2},\tab b^2=\<B,C\>=\frac{st}{2},\tab e^2=\<A,C\>=\frac{rs}{2}.
}
Continuing to abuse notation, we shall index the edge with the same label as the $\l$-length.

Let $\D CDA$ be another positively ordered triple, then $D$ is of the form 
$$D=(x_1,x_2,-y,z|\xi_1^+,\xi_2^+,\xi_1^-,\xi_2^-)$$ with $y>0$. We have
\Eq{
c^2=\<C,D\>=\frac{sx_2}{2},\tab d^2=\<A,D\>=\frac{rx_1}{2}.
}
As in the bosonic case, we can solve for the even parameters from the $\l$-lengths themselves. 
\begin{Lem}\label{rst}Denote the \emph{cross-ratio} by 
\Eq{
\mathcal{X}=\frac{ac}{bd}.
}Then we have
\Eq{
r&=\sqrt{2}\frac{ae}{b},\tab s=\sqrt{2}\frac{be}{a},\tab t=\sqrt{2}\frac{ab}{e},}
\Eq{
x_1&=\sqrt{2}\frac{bd^2}{ae}=\sqrt{2}\frac{cd}{e}\mathcal{X}\inv,\tab x_2=\sqrt{2}\frac{ac^2}{be}=\sqrt{2}\frac{cd}{e}\mathcal{X}, \tab y=\sqrt{2}\frac{cd}{e}.
}
In particular we find  
\Eq{
\sqrt{\frac{x_2}{x_1}}=\frac{ac}{bd}=\mathcal{X}.
}
\end{Lem}

Next let us calculate the odd invariants $[\s]$ of the triple $\D CDA$. 
\begin{Prop}\label{sigma} $\D CDA$ can be transformed into standard position with
\Eqn{
C&\mapsto \hat{r}(0,1,0,0|0,0,0,0),\\
D&\mapsto \hat{t}(1,1,1,\frac{\s_1\s_2}{2}|\s_1,\s_2,-\s_2,\s_1),\\
A&\mapsto \hat{s}(1,0,0,0|0,0,0,0),
}
where 
\Eq{
\hat{t}&=\sqrt{x_1x_2},\tab \hat{r}=\sqrt[4]{\frac{x_1}{x_2}}s,\tab \hat{s}=\sqrt[4]{\frac{x_1}{x_2}}r,\\
\s_1&=\sqrt[4]{\frac{x_1}{x_2}}\frac{\xi_2^-}{\sqrt{x_1x_2}}=\frac{\xi_2^-}{\sqrt{yx_2}}=-\frac{\xi_1^+}{\sqrt{yx_1}},\label{fermpara}\\
\s_2&=-\sqrt[4]{\frac{x_1}{x_2}}\frac{\xi_1^-}{\sqrt{x_1x_2}}=-\frac{\xi_1^-}{\sqrt{yx_2}}=-\frac{\xi_2^+}{\sqrt{yx_1}}.\label{fermpara2}
}
\end{Prop}
\begin{proof}
Acting on $\D CDA$ by the diagonal transformation $D_a$ with $a=\sqrt[4]{\frac{x_2}{x_1}}=\sqrt\mathcal{X}$ sends
\Eqn{C&\mapsto \hat{r}(1,0,0,0|0,0,0,0),\\
D&\mapsto  \hat{t}(1,1,-1,\frac{\s_1\s_2}{2}|-\s_1,-\s_2,-\s_2,\s_1),\\
A&\mapsto \hat{s}(0,1,0,0|0,0,0,0).
}
Applying $J$ then sends $\D CDA$ to standard position and the result follows.
\end{proof}
\begin{Cor}\label{rstCor} Using Lemma \ref{rst}, we can express all variables in terms of $\l$-lengths as follows
\Eq{
\hat{r}&=\sqrt{2}\frac{de}{c},\tab \hat{s}=\sqrt{2}\frac{ce}{d},\tab\hat{t}=\sqrt{2}\frac{cd}{e}=y,
}
\Eq{
\xi_1^-&=-\sqrt\mathcal{X}\sqrt{2}\frac{cd}{e}\s_2,\tab &&\xi_2^-=\sqrt\mathcal{X} \sqrt{2}\frac{cd}{e}\s_1,\\
\xi_1^+&=-\sqrt\mathcal{X}\inv\sqrt2\frac{cd}{e}\s_1,\tab &&\xi_2^+=-\sqrt\mathcal{X}\inv\sqrt{2}\frac{cd}{e}\s_2,\nonumber}
\Eq{
\hat{z}&=\frac{cd}{\sqrt{2}e}\s_1\s_2=z,
}
where we have used \eqref{*} from Proposition \ref{orbit_relation} for the expression of $\xi_1^+,\xi_2^+$ and $z$. 
\end{Cor}

Recall that the odd invariants are defined modulo a rescaling. Let us introduce the following terminology.
\begin{Def}\label{ratioS} A quadrilateral $\dia ABCD$ is said to be in the \emph{standard $(\h,\s)$-position with ratio $h_e$} if
\begin{itemize}
\item The ordered triple $\D ABC\in \cS^{h_e\h}$ for $h_e>0$,
\item The triangle $\D CDA$ is positively oriented, and $(J\circ D_{\sqrt\mathcal{X}}\circ Z_{h_e})\cdot\D{CDA}\in \cS^{h_e\inv \s}$.
\end{itemize}
We shall denote such a standard quadrilateral by
\Eq{
\dia ABCD\in \cS_{h_e}^{\h,\s}
} and call the transformation 
\Eq{\label{upsidedown}
\Up_{h_e}^\mathcal{X}:=J\circ D_{\sqrt\mathcal{X}}\circ Z_{h_e}
} the \emph{upside-down transformation} of the quadrilateral $\dia ABCD$. If the odd parameters $\h,\s$ have already been fixed, we shall simply say $\D ABC$ has \emph{ratio $h_e$} with $\D CDA$ across the edge $e$. 
\end{Def}
By Proposition \ref{sigma}, if $D=(x_1,x_2,-y,z|\xi_1^+,\xi_2^+,\xi_1^-,\xi_2^-)$ for $y>0$, then $( \s_1,\s_2)=(-\frac{\xi_1^+}{\sqrt{yx_1}},-\frac{\xi_2^+}{\sqrt{yx_1}})$.
Note that the $\l$-lengths uniquely define 4 points $\dia ABCD\in \cS_{h_e}^{\h,\s}$ on the light cone by Lemma \ref{rst} and Corollary \ref{rstCor}. 

\begin{Lem} \label{Stand} Let $\dia ABCD\in\cS_{h_e}^{\h,\s}$ with cross-ratio $\mathcal{X}$. Then
\begin{itemize} 
\item[(1)] $Z_{-1}\cdot \dia ABCD\in \cS_{h_e}^{-\h,-\s}$ where $-\h:=(-\h_1,-\h_2)$;
\item[(2)] $\Psi\cdot\dia ABCD\in \cS_{h_e\inv}^{\h^{op},\s^{op}}$ where $\h^{op}:=(\h_2,\h_1)$;
\item[(3)] Let $\dia A'B'C'D':=\Up_{h_e}^\mathcal{X}\cdot \dia ABCD$. Then $\dia C'D'A'B'\in \cS_{h_e\inv}^{\s,-\h}$.
\end{itemize}
\end{Lem}

\begin{proof} (1) and (2) follow from the definition of the action.

(3) By definition $\D C'D'A'\in \cS^{h_e\inv \s}$, while $B'$ is of the form
\Eqn{
B&\mapsto t(a^2,a^{-2},-1,\frac{\h_1\h_2}{2}|a\h_1,a\h_2,a\inv \h_2,-a\inv \h_1),
}
where $a=\sqrt\mathcal{X}\inv$. Hence $\s_i^{new}:=-\frac{\xi_i^+}{\sqrt{yx_1}}=-\frac{a\h_i}{\sqrt{1\cdot a^2}}=-\h_i$ or $\s_i^{new}=-\h_i$.

\end{proof}
\begin{Prop}\label{ratiounique} If $\h\in[\D ABC]$ and $\s\in[\D CDA]$ for two positively ordered triples $\D ABC$ and $\D CDA$, then there exists $g\in \SL(1|2)$ such that $g\cdot \dia ABCD\in \cS_{h_e}^{\h,\til{\s}}$ for some $h_e>0$, where $\til{\s}=\pm\s$ or $\pm\s^{op}$. The choice of $\til{\s}$ is uniquely determined, while $h_e$ is uniquely defined by $\h$ and $\s$ up to multiplication by constants $c_1$ and $c_2$ such that $c_1\h=\h$ and $c_2\s=\s$.
\end{Prop}
\begin{proof} By Proposition \ref{orbit}, there exists $g\in \SL(1|2)$ such that $g\cdot \D ABC\in \cS^{\h}$. Then $(\Up_{h_e}^\mathcal{X}\circ g)\cdot \D CDA$ is in standard position hence belongs to $\cS^{\s'}$ for some $\s'\in [\D CDA]$ and therefore equals $a\s$ or $a\s^{op}$ with $a\neq 0$. By Proposition \ref{stab}, the stabilizer for $\D ABC$ is given by $Z_c$ for some $c>0$, hence the sign of $a$ and the order $\s$ or $\s^{op}$ is determined uniquely. Let $a=\pm h_e\inv$ with $h_e>0$ and $\til{\s}:=h_e\s'$,  applying $Z_{h_e}\inv\in SL(1|2)$ we see that $\dia ABCD\in \cS_{h_e}^{\h,\til{\s}}$ as required.

Finally, if $\dia ABCD\in\cS_{h_e}^{\h,\s}$, and $Z_c\cdot \dia ABCD\in \cS_{h_e'}^{\h,\s}$, then $ch_e\h = h_e'\h$ and $c\s = \s$, which implies $\frac{h_e'}{h_e}=c_1c_2$ satisfying the condition above.
\end{proof}

\begin{Cor}The stabilizers of $\dia ABCD\in \cS_{h_e}^{\h,\s}$ are of the form $Z_{c}$ for some even element $c$ such that $c-1$ is an annihilator of $\theta_i$, $\sigma_i$, $i=1,2$.
\end{Cor}
\begin{proof} From the proof above, we have $c\h=\h$ and $c\s=\s$, hence the conclusion follows.
\end{proof}
%==============================================================================

\section{Spin structures and fatgraph connections}\label{sec:spin}

Let $\t$ be a trivalent fatgraph spine that is dual to some triangulation $\D$ of the surface $F:=F_g^s$ determined up to isotopy. Let $\t_0,\t_1$ denote the set of vertices and edges of $\t$ respectively. In particular each vertex $v\in\t_0$ corresponds to a triangle in $\D$.  Let $\w$ be an orientation on the edges $\t_1\subset\t$. As in \cite{PZ}, we define a \emph{fatgraph reflection} at a vertex $v$ of $(\t,\w)$ to reverse the orientations of $\w$ on every edge of $\t$ incident to $v$.
\begin{Def}\label{fatgraph} We define $\cO(\t)$ to be the equivalence classes of orientations on a trivalent fatgraph $\tau$ spine of $F$, where the equivalence relation is given by $\w_1\sim \w_2$ iff $\omega_1$ and $\omega_2$ differ by finite number of fatgraph reflections. It is an affine $H^1$-space where the cohomology group $H^1:=H^1(F;\Z_2)$ acts on $\cO(\t)$ by changing the orientation of the edges along cycles.
\end{Def}

In \cite{PZ} various realizations of the spin structures on the surface $F$ are described. 
\begin{Def}\cite{Nat} A spin structure is determined by a lift $\til{\rho}:\pi_1(F)\to SL(2,\R)$ of the Fuchsian representation $\rho:\pi_1(F)\to PSL(2,\R)$  familiar from Teichm\"uller theory.
\end{Def}

In fact, given a spin structure, one can define a quadratic form $q:H_1(F;\Z_2)\to \Z_2$ on simple cycles $\c\in\pi_1(F)$ by
\Eq{\label{quadform}
q([\c])=\sign ~Tr(\til{\rho}(\c)),
}
and extend to all of $H_1(F;\Z_2)$ by $q(a+b)=q(a)+q(b)+a\cdot b$ where $a\cdot b$ denotes the intersection form. Then we have 

\begin{Prop}\label{isospin}\cite{CR, PZ} The set of spin structures is isomorphic to the space of quadratic forms $\cQ(F)$ on $H_1(F;\Z_2)$, and also isomorphic to $\cO(\t)$, as affine $H^1$-spaces. 
\end{Prop}

In terms of oriented fatgraphs, under the chosen isomorphism in \cite{PZ} between $\cO(\t)$ and $\cQ(F)$, the evolutions for spin structure under a change of the triangulations generated by flips that preserves the corresponding quadratic forms are described. By choosing appropriate representatives of the orientation classes, the flip transformation is equivalent to the rule depicted in Figure \ref{flipgraph}, i.e., the middle arrow changes from pointing up to pointing left, and the upper right arrow changes orientation, while the other branches remain the same.

%%%%%%%%%%     NICE PICTURE   %%%%%%%%%%%%%% 
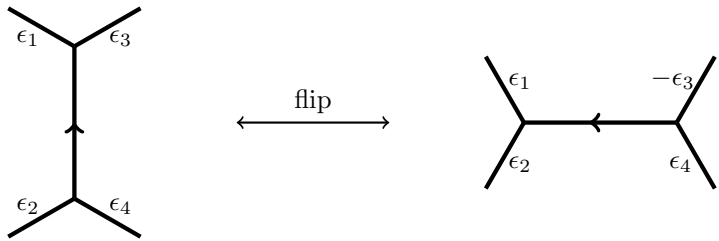
\begin{figure}[h!]
\begin{center}

\begin{tikzpicture}[ultra thick, baseline=1cm]
\draw (0,0)--(210:1) node[above] at (210:0.7){$\e_2$};
\draw (0,0)--(330:1) node[above] at (330:0.7){$\e_4$};
    % draw the connecting line
    \draw[ 
 	ultra thick,
        decoration={markings, mark=at position 0.5 with {\arrow{>}}},
        postaction={decorate}
        ]
        (0,0) -- (0,2);
\draw[yshift=2cm] (0,0)--(30:1) node[below] at (30:0.7){$\e_3$};
\draw[yshift=2cm] (0,0)--(150:1) node[below] at (150:0.7){$\e_1$};
\end{tikzpicture}
\begin{tikzpicture}[baseline]
\draw[<->, thick](0,0)--(2,0);
\node[above] at (1,0) {flip};
\node at (-1,0){};
\node at (3,0){};
\end{tikzpicture}
\begin{tikzpicture}[ultra thick, baseline]
\draw (0,0)--(120:1) node[above] at (100:0.3){$\e_1$};
\draw (0,0)--(240:1) node[below] at (260:0.3){$\e_2$};
    % draw the connecting line
    \draw[ 
        decoration={markings, mark=at position 0.5 with {\arrow{<}}},
        postaction={decorate}
        ]
        (0,0) -- (2,0);
\draw[xshift=2cm] (0,0)--(60:1) node[above] at (100:0.3){$-\e_3$};
\draw [xshift=2cm](0,0)--(-60:1) node[below] at (-80:0.3){$\e_4$};
\end{tikzpicture}
\end{center}
\caption{Spin graph evolution with $\e_i$ denote orientations}
\label{flipgraph}
\end{figure}
%%%%%%%%%%%%%%%%%%%%%%%%%

A slight modification of the rule has to be made if some of the branches are identified, e.g., in the cases $F_1^1$ and $F_0^3$. This will be described in Appendix \ref{sec:specialPtolemy}.

In the next section, starting from an element $s\in \cO(\t)$, we shall construct the one-to-one correspondence providing the lift $\til{\rho}$ as above. Henceforth we shall also refer to elements of $\cO(\t)$ as spin structures as well.

\begin{Def}
Fix a fatgraph $\t$, we denote by $o_{\w}(e)$ the orientation of the edge $e\in\t$ in the orientation $\w$. We define $\d_{\w_1,\w_2}:\t_1\to \Z_2$ by 
\Eqn{
\d_{\w_1,\w_2}(e):=\case{+1& o_{\w_1}(e)=o_{\w_2}(e),\\-1& o_{\w_1}(e)\neq o_{\w_2}(e),}
}
which defines an element in $H^1(F;\Z_2)$.
\end{Def}

\begin{Def}\label{graphconnect}\cite{P2} Let $G$ be a group. A \emph{$G$-graph connection} on $\t$ is the assignment $h_e\in G$ to each oriented edge $e$ of $\t$ so that $h_{\over[e]}=h_e\inv$ if $\over[e]$ is the opposite orientation to $e$. Two assignments $\{h_e\},\{h_e'\}$ are equivalent iff there are $t_v\in G$ for each vertex $v$ of $\t$ such that $h_e'=t_v h_e t_w\inv$ for each oriented edge $e\in\t_1$ with initial point $v$ and terminal point $w$.
\end{Def}

\noindent As an immediate consequence of Proposition \ref{isospin}, we obtain

\begin{Cor}
The space of spin connections on $F$ is identified with $\mathbb{Z}_2$-graph connections on a given fatgraph $\tau$ of $F$.
\end{Cor}
\begin{proof}
Pick an orientation $\omega$ on fatgraph $\tau$. For any orientation $\omega'$, we label the edge by $+1$ if they agree on that edge and by $-1$ if they do not agree. Classes of orientations give rise to classes of $\mathbb{Z}_2$-connections, since the reflection of orientation at the vertices corresponds to multiplication by $-1$.
\end{proof}

\noindent For a Lie (super) group $G$, the following statement is true (see e.g., \cite{Dar}). 

\begin{Prop}\label{isoflat} The moduli space of flat $G$-connections on $F$ is isomorphic to the space of equivalence classes of $G$-graph connections on $\t$.
\end{Prop}

The case which will be exploited fully in the next section is when $G=\R_+$, i.e., the set of all even elements in $S$ (underlying the Grassmann algebra over $\mathbb{R}$) with positive bodies.

Since the product of $h_e$ around each cycle gives the monodromies defining the flat connection, the proposition implies that one can always rescale $h_e$ around vertices in such a way such that if $a,b,e$ are three edges oriented towards a vertex $v$, we have $h_a h_b h_e=1$, and this uniquely defines the set of equivalence classes of graph connections on $\t$. In particular for $F_g^s$ it is known that the moduli space of flat $G$-connections has dimension $(2g+s-1)\dim(G)$.
%==============================================================================
\section{Decorated super-Teichm\"uller space}\label{sec:teich}
In this section we describe the coordinates of the decorated super-Teichm\"uller space using the basic calculation from Section \ref{sec:basic}. 

We fix a surface $F:=F_g^s$ with genus $g\geq 0$ and $s\geq 1$ punctures where $2g+s-2>0$. We fix a triangulation $\D$ of $F$ and the corresponding trivalent fatgraph $\t\subset F$. We consider the lift of $\D$ to an ideal triangulation $\til{\D}$ with corresponding fatgraph $\til{\t}$ of the universal cover $\pi:\til{F}\to F$ where $\til{F}$ is topologically equivalent to the unit disk with hyperbolic metric, and where the ideal vertices $\til{\D}_\oo$ of $\til{\D}$ lie on the boundary circle $S^1$ at infinity.

\begin{Def}Given $F,\D$ as above. We define the coordinate system $\til{C}(F,\D)$ as follows:
\begin{itemize}\item we assign to each edge $e$ of $\D$ a positive even coordinate called the $\l$-length, also denoted by $e$;\\
\item we assign to each triangle of $\D$ two ordered odd coordinates $(\h_1,\h_2)$;\\
\item we assign to each edge $e$ of a triangle of $\D$ a positive even coordinate $h_e$, called the \emph{ratios}, such that if $h_e$ and $h_e'$ are assigned to two triangles sharing the same edge $e$ we have $h_e h_e'=1$.
\end{itemize}
The odd coordinates are defined up to an overall sign changes $\h_i\corr -\h_i$, as well as an overall involution $(\h_1,\h_2)\corr (\h_2,\h_1)$.
\end{Def}

Note that the ratios $\{h_e\}$ uniquely define an $\R_+$-graph connection on $\t$, where the oriented edge of $\t$ pointing towards a triangle crossing the edge $e$ will have the value $h_e$ of that triangle assigned.

\begin{Def} Let $\vec[c]\in \til{C}(F,\D)$ be a coordinate vector. If $h_a,h_b,h_e$ are ratios assigned to a triangle $T$ with odd coordinates $(\h_1,\h_2)$, then a \emph{vertex rescaling at $T$} of $\vec[c]$ is the new coordinate vector obtained by changing 
\Eq{
(h_a,h_b,h_e,\h_1,\h_2)\mapsto (u h_a,u h_b,u h_e,u\inv\h_1,u\h_2)
} for some $u>0$, and all other coordinates fixed. Two coordinate vectors of $\til{C}(F,\D)$ are said to be equivalent if they are related by a finite number of vertex rescalings. In particular the underlying graph connections on $\t$ are equivalent (cf. Definition \ref{graphconnect}).
\end{Def}

 Let $C(F,\D):=\til{C}(F,\D)/\sim$ be the equivalence classes of coordinate vectors. Then by Proposition \ref{isoflat} it can be represented by coordinates with $h_a h_b h_e=1$ for the ratios of the same triangle. We can easily deduce that
\Eq{
C(F,\D)\simeq \R_+^{8g+4s-7|8g+4s-8}/\Z_2\x\Z_2
} in accordance with the dimension given in \cite{Nat} (with an extra $s$-dimension given by decorations). The coordinates naturally pullback to coordinates on the ideal triangulations $\til{\D}$ of the universal cover $\til{F}$ by $\pi$.

\begin{Def} For each fatgraph $\t\subset F$ and each spin structure $s\in \cO(\t)$, we shall fix one representative orientation which will be denoted by $\w_{s,\t}$ such that $[\w_{s,\t}]=s\in\cO(\t)$.
\end{Def}

As we have seen in Proposition \ref{ratiounique}, the transformations involving ordered triples and quadrilaterals in $\cL_0$ possess some freedom involving the $\cZ$-subgroup elements. In order to fix this degree of freedom (which will be absorbed into the ratios $h_e$), we introduce the following definition.
\begin{Def} Given a coordinate vector $\vec[c]\in \til{C}(F,\D)$, a transformation $g\in \SL(1|2)$ is called \emph{$\vec[c]$-admissible} if $g$ is a composition only of 
\begin{itemize}
\item prime transformations $P_{h_a,h_b}^{\h}$;
\item the upside-down transformations $\Up_{h_e}^\mathcal{X}:=J\circ D_{\sqrt\mathcal{X}}\circ Z_{h_e}$;
\item sign changes $J^2=Z_{-1}$ and involution $\Psi$,
\end{itemize}
where $h_a,h_b$ are ratios of the triangle with odd coordinates $\h$ with $\mathcal{X}$ and $h_e$ the respective cross-ratio and ratio of the diagonal of the quadrilateral to be affected by the upside-down transformation.
\end{Def}

Now we can state the first main result of the paper.

\begin{Thm}\label{main_lift} Fix $F,\D,\t$ as before. Let $\w_{\sigma}:=\w_{s_{\sigma},\t}$ correspond to a specified spin structure $s_{\sigma}$ of $F$, and let $\w_{\iota}:=\w_{s_{\iota},\t}$ be the representative of another spin structure $s_{\iota}$. Given a  coordinate vector $\vec[c]\in \til{C}(F,\D)$, there exists a map called the \emph{lift},
\Eqn{
\ell_{\w_{\sigma},\w_{\iota}}: \til{\D}_\oo \to \cL_0,} such that if $\dia ABCD$ is the image of a quadrilateral from $\til{\D}_\oo$ with coordinates labeled as in Figure \ref{basic}, $\D ABC$ has ratio $h_e$ with $\D CDA$ across the edge $e$, and $t$ is the edge of the fatgraph $\t$ with orientation $\w$ pointing from $\D CDA\to\D ABC$, then 
\Eq{\label{liftingcase}t:\D_{ABC}\tto \D_{CDA}\=>\case{g\cdot \dia ABCD\in \cS_{h_e}^{\h,\s}&  \mbox{ if }\d_{\w_{\sigma},\w_{\iota}}(t)=1,\\g\cdot \dia ABCD\in \cS_{h_e}^{\h,\s^{op}}& \mbox{ if }\d_{\w_{\sigma},\w_{\iota}}(t)=-1}}
for some $g\in \SL(1|2)$ that is $\vec[c]$-admissible.

The lift is uniquely determined up to post-composition by $\SL(1|2)$, and only depends on the equivalence classes $C(F,\D)$ of the coordinates.
\end{Thm}
\begin{Thm}\label{main_group} Fixing a coordinate vector $\vec[c]\in C(F,\D)$ and given a lift $\ell:=\ell_{\w_{\sigma},\w_{\iota}}:\til{\D}_\oo \to \cL_0$ constructed as in Theorem \ref{main_lift}, there is a representation $\what{\rho}:\pi_1:=\pi_1(F)\to \SL(1|2)$ uniquely determined up to conjugacy by an element of $\SL(1|2)$ such that
\begin{itemize}\item[(1)] $\ell$ is $\pi_1$-equivariant, i.e., $\what{\rho}(\c)(\ell(a))=\ell(\c(a))$ for each $\c\in\pi_1$ and $a\in \til{\D}_\oo$;\\
\item[(2)] $\what{\rho}$ is a super Fuchsian representation, i.e., the natural projection $$\rho:\pi_1\xto{\what{\rho}} \SL(1|2) \to SL(2,\R)\to PSL(2,\R)$$ is a Fuchsian representation;\\
\item[(3)] the lift $\til{\rho}:\pi_1\xto{\what{\rho}} \SL(1|2) \to SL(2,\R)$ of $\rho$ does not depend on $\w_{\iota}$, and the space of all such lifts is in one-to-one correspondence with the spin structures $[\w_{\sigma}]\in\cO(\t)$.
\end{itemize}
\end{Thm}
\begin{Def} The space of $\SL(1|2)$-orbits of lifts $\ell:\til{\D}_\oo\to \cL_0$ that is $\pi_1$-equivariant for some super Fuchsian representation $\what{\rho}:\pi_1\to \SL(1|2)$ is called the decorated super Teichm\"uller space $S\til{T}(F)$.
\end{Def}

\begin{proof1} The construction is similar to the one presented in \cite{PZ} but with some modifications. In particular our construction of the lift here will directly incorporate the spin structures by determining explicitly the signs and orders of the odd coordinates $\h=(\h_1,\h_2)$ for each triangle.

First we note that by Lemma \ref{Stand}, if $t$ has the opposite orientation instead, then \eqref{liftingcase} implies
\Eq{\label{liftingcase2}
t:\D_{ABC}\to \D_{CDA}\=>\case{g\cdot \dia ABCD\in \cS_{h_e}^{\h,-\s}&  \mbox{ if }\d_{\w_{\sigma},\w_{\iota}}(t)=1,\\g\cdot \dia ABCD\in \cS_{h_e}^{\h,-\s^{op}}& \mbox{ if }\d_{\w_{\sigma},\w_{\iota}}(t)=-1}
}
for some admissible $g\in \SL(1|2)$.

To begin our construction, we first choose a distinguished triangle-edge pair $(T,e)$ in $\til{\D}_\oo$ with specified $\l$-lengths, odd coordinates $\h$, and ratio $h_e$. Then according to Lemma \ref{rst} there is a unique triangle $\D ABC\subset \cL_0$ realizing the $\l$-lengths and satisfying $\D ABC\in \cS^{h_e\h}$ with $BC$ the image of the distinguished edge. We define $\D ABC$ to be the lift of $T$ and call $T$ the \emph{base triangle} of our lift. 

Consider the adjacent triangle $T'$ opposite to $e$ with odd coordinates $\s$. Then we define $\D CDA$ to be the lift of $T'$, where $D\in \cL_0$ is the unique point by Proposition \ref{sigma} that realizes the $\l$-lengths, and satisfies \eqref{liftingcase} or \eqref{liftingcase2} with $g=1$. We define $\D CDA$ to be the image of $T'$.

Now consider the lift $\D AD'B$ of the adjacent triangle of $T$ with odd coordinates $\s'$ and ratio $h_a$ at the edge $AB$ of $\D ABC$. Applying $P_{h_e,h_a}^\h$, we bring $\D BCA$ to standard position. Then the point $D'\in \cL_0$ is uniquely determined by $\dia BCAD'\in \cS_{h_a}^{\h,\til{\s'}}$ where $\til{\s'}=\pm \s'$ or $\pm {\s'}^{op}$ according to \eqref{liftingcase} or \eqref{liftingcase2}. If we now apply the upside-down transformation $\Up_{h_a}^{\mathcal{X}'}$ where $\mathcal{X}'$ is the cross-ratio for $\dia BCAD'$, and if necessary the sign change $J^2$ and involution $\Psi$, we bring $\D AD'B\in \cS^{h_a\inv\s'}$ and we can continue our inductive process of lifting other triangles. 

More precisely, fixing $T$ as the base triangle of our lift, the lift $\D A'B'C'$ of a triangle $T'$ with specified $\l$-lengths, odd coordinates $\h'$ and ratio $h'$ across $A'C'$ is defined to be $g\inv\cdot \D_{\h'}$, where $\D_{\h'}\in \cS^{h'\h'}$ is determined by the unique ordered triple in $\cL_0$ with specified fermions and $\l$-lengths of $T'$, and $g\in \SL(1|2)$ is the unique admissible transformation constructed above bringing $\D A'B'C'$ to standard position along the path of $\til{\t}$ joining $T$ and $T'$ in the universal cover.

We completely define the lift $\ell:\til{\D}_\oo\to \cL_0$ in this manner, which satisfies the statement of the theorem
by construction. Choosing a different base triangle and edge pair amounts to an overall action by $g\in \SL(1|2)$ on $\cL_0$, where $g$ is the unique admissible transformation that brings the image of this new triangle in the original lift to standard position, i.e., $\ell_{new} = g\cdot \ell$. Hence the lift $\ell$ is uniquely defined up to post-composition with an $\SL(1|2)$ element.

Finally, under a vertex rescaling $\a$ of the coordinate at some triangle $T_\h$ that is not the base triangle, the admissible transformation along the portion of the path $T_{\s'} \to T_\h \to T_{\s''}$ that passes through $T_\h$ changes as (cf. \eqref{leftright})
\Eqn{
g_\a&=...\circ P_{\a\inv h_e\inv, h_t}^{\s'',\pm} \circ \Up_{\a h_e}^{\mathcal{X}} \circ P_{\a h_a, \a h_e}^{\a\inv\h,\pm} \circ \Up_{\a\inv  h_a\inv}^{\mathcal{X}'} \circ P_{h_s,\a\inv h_a\inv}^{\s',\pm}\circ...\\
&=...\circ P_{h_e\inv, h_t}^{\s'',\pm}  Z_\a\inv\circ Z_\a\Up_{h_e}^{\mathcal{X}} \circ P_{\a h_a, \a h_e}^{\a\inv\h,\pm} \circ \Up_{h_a\inv}^{\mathcal{X}'}Z_{\a}\inv \circ Z_\a P_{h_s, h_a\inv}^{\s',\pm}\circ...\\
&=...\circ P_{h_e\inv, h_t}^{\s'',\pm} \circ \Up_{h_e}^{\mathcal{X}} \circ P_{h_a, h_e}^{\h,\pm} \circ \Up_{h_a}^{\mathcal{X}'}\circ P_{h_s, h_a}^{\s',\pm}\circ...\\
&=g,
}
where we used the fact that $P_{\a h_a, \a h_e}^{\a\inv\h,\pm}=P_{h_a, h_e}^{\h,\pm}$. Hence the definition of our lift does not change. The case with sign changes and involutions are similar. However if $T_\h$ is our base triangle, then the lift differs by an overall scaling of $Z_{\a}\in \SL(1|2)$.
\end{proof1}
\qed

\begin{proof2} The construction is a simplified version of \cite{PZ} without the need to modify the signs of $\what{\rho}$ for the spin structures since they are already incorporated in the construction of the lift.

As in \cite{PZ}, we fix a base triangle $T$ for our lift and choose a connected fundamental domain $\bD\subset \til{F}$ for the action of $\pi_1$ that contains $T$. Then $\bD$ is a $4g+2s$-sided ideal polygon, and in particular, the frontier edges of $\bD$ in $\til{F}$ arise in pairs $c, c'$ together with an abstract identification $c'=\c(c)$ induced by some $\c\in\pi_1$. We let $c_i'=\c_i(c_i)$ enumerate the collection of these edges pairing, where $i=1,..., 2g+s$. Then it is known that $\pi_1$ is the free group generated by these $2g+s$ elements $\c_i$.

To determine the image of $\what{\rho}(\c_i)\in \SL(1|2)$, let $\D ABC$ and $\D A'B'C'$ be the lift of the unique pair of triangles such that $BC=\ell(c_i)$, $B'C'=\ell(c_i')$, $\ell\inv(ABC)\sub \bD,$ and $\ell\inv(A'B'C')\not\subset \bD$. Then by definition of $\ell$, there are unique admissible transformations $g,g'$ bringing standard position from $\D ABC$ and $\D A'B'C'$ respectively to $\ell(T)$. We define 
\Eq{
\what{\rho}(\c_i):={g'}\inv g\in \SL(1|2).
} Since $\pi_1$ is a free group this defines our representation $\what{\rho}$. 

More explicitly, let $\c_i$ be homotopically represented by a path in $\t$. Then $\what{\rho}(\c_i)$ is a composition of the form
\Eq{\label{explicit}
\what{\rho}(\c_i)=\prod_k Z_{\pm1}\circ \Psi_\pm\circ \Up_{h_k'}^{\mathcal{X}_k} \circ P_{h_k,h_k'}^{\h_k,\pm}\in \SL(1|2)
}
that brings standard position from $\D ABC$ to $\D A'B'C'$ moving along the edge of $\t$, where $P_{h_k,h_k'}^{\h_k,\pm}$ is the prime transformation corresponding to turning left $(+)$ or right $(-)$ at the vertex of $\t$ with starting ratio $h_k$ and ending ratio $h_k'$. The signs $Z_{\pm1}$ and involutions $\Psi_\pm\in\{Id, \Psi\}$ correspond to the orientations $\w_{\sigma}$ and $\w_{\iota}$ respectively according to the rule \eqref{liftingcase} and \eqref{liftingcase2}.

The same argument as in \cite{PZ} shows that choosing a different initial triangle defining our lift $\ell$, and choosing a different fundamental domain $\bD'$ for our construction, amounts again to an overall conjugation of $\what{\rho}$ by an admissible element of $\SL(1|2)$.

(1) follows directly from the definition of the construction of $\what{\rho}$. By an argument in \cite{P1,P2}, (2) follows from the fact that the bosonic reduction $\rho(\pi_1)\subset PSL(2,\R)$ by construction leaves invariant the tessellation $\til{\D}\subset \til{F}\simeq \mathbb{D}$.

Finally to show (3), recall that the projection $\SL(1|2)\to SL(2,\R)$ maps $\Psi\mapsto 1$, hence $\til{\rho}$ does not depend on $\w_{\iota}$. Now we note that if $\w$ and $\w'$ are related by a fatgraph reflection at the vertex $v$, then the orientation of either zero or two edges of $\t$ are flipped along each cycle. Since the projection $Z_{-1}\to(-1)$ to $SL(2,\R)$ commutes with every operators in \eqref{explicit}, the projection $\til{\rho}(\c_i)\in SL(2,\R)$ remains unchanged under fatgraph reflection. In particular if $[\w]$ and $[\w']$ defines the same lift $\til{\rho}$ to $SL(2,\R)$, then all cycles of $\pi_1$ differ in orientations at even number of edges of $\t$, which implies $[\w]=[\w']\in\cO(\t)$, thus proving the one-to-one correspondence between lifts $\til{\rho}$ and spin structures.
\end{proof2}\qed
\begin{Cor} Given an oriented simple cycle $\c\in\pi_1(F)$ homotopic to a path on the fatgraph $\t$ with orientation class $[\w]$, the quadratic form corresponding to the lift $\til{\rho}:\pi_1\to SL(2,\R)$ according to \eqref{quadform} is given by
\Eq{\label{ourquad}
q([\c])=(-1)^{L_\c}(-1)^{N_\c}=(-1)^{R_\c}(-1)^{\over[N]_\c},}
where $L_\c$ (resp. $R_\c$) is the number of left (resp. right) turns of $\c$ on the fatgraph $\t$, and $N_\c$ (resp. $\over[N]_\c$) is the number of edges of $\t$ such that $\c$ and $\w$ have the same (resp. opposite) orientation.
Moreover under the flip transformation, the spin structure changes according to the rule expressed in the Figure \ref{flipgraph}.
\end{Cor}
\begin{proof} According to \eqref{quadform}, we need to calculate the sign of the trace of $\til{\rho}(\c)$. By the explicit expression \eqref{explicit}, under the projection to $SL(2,\R)$, the $Z_{h_k}$ elements for $h_k>0$ become $h_k I_{2\x2}$ hence do not affect the sign. Therefore the projection of the representation of $\til{\rho}(\c)$ is given by $N_\c$ number of signs $\veca{-1&0\\0&-1}$ and compositions of left and right turns:
\Eqn{
L_k&:=\veca{-1&1\\-1&0}\veca{0&1\\-1&0}\veca{a_k&0\\0&a_k\inv}=\veca{-a_k&-a_k\inv\\0&-a_k\inv},\\
R_k&:=\veca{-1&1\\-1&0}\inv \veca{0&1\\-1&0}\veca{a_k&0\\0&a_k\inv}=\veca{a_k&0\\a_k&a_k\inv},
}
where $a_k>0$. Then it can easily be shown by induction that multiplication by $L_k$ changes the sign of trace while $R_k$ does not, hence the trace depends on the number of left turns $L_\c$ made by $\c$. Finally letting $|\c|$ be the length of the cycle, we have $R_\c=|\c|-L_\c$ and $N_\c=|\c|-\over[N]_\c$, hence the second equation follows.

In order to prove the last part of the statement, one just needs to follow the same route as in \cite{PZ}, namely  
consider various pieces of curves running along the part of the fatgraph which is affected by the flip and compare the contributions to the quadratic form from those before and after the flip. 
\end{proof}
\begin{Rem} One can show that using the \emph{special Kasteleyn orientation} and the \emph{canonical dimer} described in \cite{PZ}, the formula in \cite{CR} produces a quadratic form that differs from \eqref{ourquad} by an overall change of signs for each simple cycle. Hence if we modify the definition of $\what{\rho}(\c_i)$ by composition with $Z_{-1}$, our construction will agree with the isomorphisms in Proposition \ref{isospin} described previously in the literature.
\end{Rem}
\begin{Thm} \label{compn} The components of $S\til{T}(F)$ are determined by two spin structures $s_{\sigma},s_{\iota}\in \cO(\t)$, thus making the total number of components $2^{2(2g+s-1)}$. For fixed representatives of the spin structures, $C(F,\D)$ provides an analytic homeomorphism from each component to $\R_+^{8g+4s-7|8g+4s-8}/\Z_2\x\Z_2$. 
\end{Thm}
\begin{proof} 
By the uniqueness of the basic calculation from Lemma \ref{rst} and Proposition \ref{sigma}, given a lift $\ell\in S\til{T}(F)$ for fixed representatives of the spin structures, the coordinates are uniquely determined since they are defined intrinsically in $\cL_0$.

If $\w_\sigma$ and $\w_\iota$ are representatives of $s_{\sigma}$ (resp.\ $s_{\iota}$) that differ by a fatgraph reflection, then from the construction of the lift $\ell$ we see that changing the coordinates from $C(F,\D)_{\w_\sigma}$ to $C(F,\D)_{\w_\iota}$ by $\h\to -\h$ (resp. $\h\to \h^{op}$) of the corresponding vertex defines the same lift. Hence the components of $S\til{T}(F)$ depends only on the equivalence classes $s_{\sigma}$ and $s_{\iota}$.
\end{proof}

\begin{Rem}
Removal of the decoration, i.e., the construction of the coordinates on the $\cN=2$ super-Teichm\"uller space follows the bosonic case. From the  construction of the lift, we obtain that the representation $\what{\rho}$ of $\widetilde{SL}(1|2)$ depends on $\lambda$-lengths only through cross-ratios $\mathcal{X}$. Therefore, passing from $\lambda$-lengths to cross-ratios for every edge, we obtain a coordinate system for $ST(F)$, either as a subspace determined by one constraint for each puncture of $F$ or more generally for $F$ as a surface with holes.
\end{Rem}
\begin{Rem}
One can obtain the coordinates on the (decorated) $\cN=1$ Teichm\"uller space by considering lifts such that $h_{e}=1$ for every edge $e$ and $\theta_1=\theta_2$ for every triangle. That precisely corresponds to the reduction $\SL(1|2)\to OSp(1|2)$. 
\end{Rem}

%==============================================================================

\section{Ptolemy transform}\label{sec:ptolemy}
In this section, we describe the Ptolemy transform on the coordinate system $\til{C}(F,\D)$ of $S\til{T}(F)$.  In particular, our result gives the mapping class group action on the super Teichm\"uller space. More precisely, 

\begin{Def} If $\D$ and $\D'$ differ by a flip of the diagonal in one quadrilateral, and the underlying fatgraphs $\w_{\sigma},\w_{\iota}$ evolve according to the rule given by Figure \ref{flipgraph}, then the local coordinate transformation $\til{C}(F,\D)\to \til{C}(F,\D')$ around the quadrilateral:
\Eq{
&(...,a,b,c,d,e, h_a, h_b, h_c, h_d, h_e,... |..., \h_1,\h_2,\s_1,\s_2, ...) \mapsto\\
&\tab  (...,a,b,c,d,f, h_a', h_b', h_c', h_d', h_f,... |..., \mu_1,\mu_2,\nu_1,\nu_2, ...)
}
(and all other coordinates fixed) is called a \emph{Ptolemy transform} if the two coordinate vectors produce the same lift $\ell\in S\til{T}(F)$.
\end{Def}

Since the definition of flip depends only on the equivalence classes $C(F,\D)$ of the coordinate vectors, the Ptolemy transform descends to a coordinate transformation $C(F,\D)\to C(F,\D')$. Since all the ratios under the projection from $\til{C}(F,\D)\to C(F,\D)$ will be rescaled, the transformation is no longer local. 

As we shall see, the calculations for the $\l$-lengths are similar to \cite{PZ}. However, the calculation of the odd coordinates and the ratios are significantly different due to the rescaling properties.

%%%%%%%    NICE PTOLEMY PICTURES  %%%%%%%%

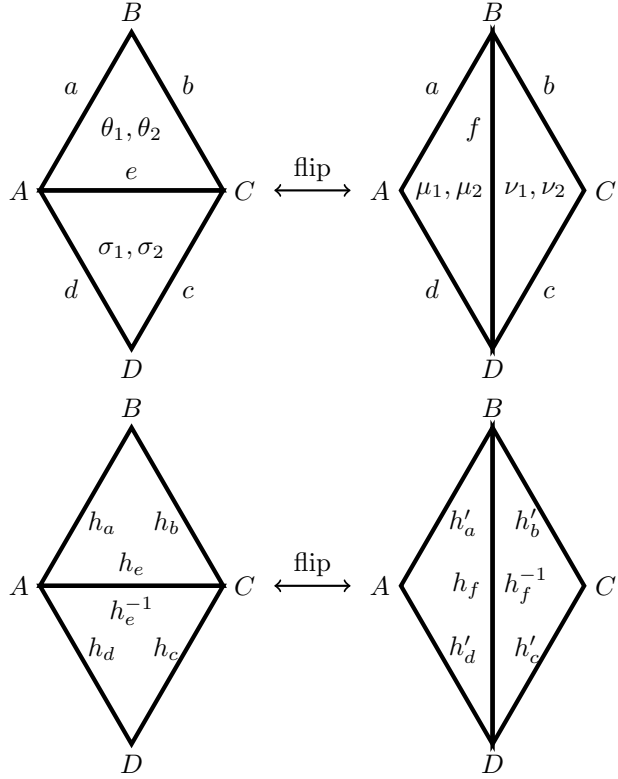
\begin{figure}[h!]
\centering
\begin{tikzpicture}[scale=0.8, baseline,ultra thick]
\draw (0,0)--(3,0)--(60:3)--cycle;
\draw (0,0)--(3,0)--(-60:3)--cycle;
\draw node[above] at (70:1.5){$a$};
\draw node[above] at (30:2.8){$b$};
\draw node[below] at (-30:2.8){$c$};
\draw node[below=-0.1] at (-70:1.5){$d$};
\draw node[above] at (1.5,0){$e$};
\draw node[left] at (0,0) {$A$};
\draw node[above] at (60:3) {$B$};
\draw node[right] at (3,0) {$C$};
\draw node[below] at (-60:3) {$D$};
\draw node at (1.5,1){$\h_1,\h_2$};
\draw node at (1.5,-1){$\s_1,\s_2$};
\end{tikzpicture}
\begin{tikzpicture}[baseline]
\draw[<->, thick](0,0)--(1,0);
\node[above]  at (0.5,0) {flip};
\end{tikzpicture}
\begin{tikzpicture}[scale=0.8, baseline,ultra thick]
\draw (0,0)--(60:3)--(-60:3)--cycle;
\draw (3,0)--(60:3)--(-60:3)--cycle;
\draw node[above] at (70:1.5){$a$};
\draw node[above] at (30:2.8){$b$};
\draw node[below] at (-30:2.8){$c$};
\draw node[below=-0.1] at (-70:1.5){$d$};
\draw node[left] at (1.5,1){$f$};
\draw node[left] at (0,0) {$A$};
\draw node[above] at (60:3) {$B$};
\draw node[right] at (3,0) {$C$};
\draw node[below] at (-60:3) {$D$};
\draw node at (0.8,0){$\mu_1,\mu_2$};
\draw node at (2.2,0){$\nu_1,\nu_2$};
\end{tikzpicture}\\
\begin{tikzpicture}[scale=0.8, baseline,ultra thick]
\draw (0,0)--(3,0)--(60:3)--cycle;
\draw (0,0)--(3,0)--(-60:3)--cycle;
\draw node[right] (ha) at (60:1.2){$h_a$};
\draw node[right =0.2 of ha] {$h_b$};
\draw node[right] (hd) at (-60:1.2){$h_d$};
\draw node[right =0.2 of hd] {$h_c$};
\draw node[above] at (1.5,0){$h_e$};
\draw node[below] at (1.5,0){$h_e\inv$};
\draw node[left] at (0,0) {$A$};
\draw node[above] at (60:3) {$B$};
\draw node[right] at (3,0) {$C$};
\draw node[below] at (-60:3) {$D$};
\end{tikzpicture}
\begin{tikzpicture}[baseline]
\draw[<->, thick](0,0)--(1,0);
\node[above]  at (0.5,0) {flip};
\end{tikzpicture}
\begin{tikzpicture}[scale=0.8, baseline,ultra thick]
\draw (0,0)--(60:3)--(-60:3)--cycle;
\draw (3,0)--(60:3)--(-60:3)--cycle;
\draw node[right] (ha) at (60:1.2){$h_a'$};
\draw node[right =0.2 of ha] {$h_b'$};
\draw node[right] (hd) at (-60:1.2){$h_d'$};
\draw node[right =0.2 of hd] {$h_c'$};
\draw node[left] at (1.5,0){$h_f$};
\draw node[right] at (1.5,0){$h_f\inv$};
\draw node[left] at (0,0) {$A$};
\draw node[above] at (60:3) {$B$};
\draw node[right] at (3,0) {$C$};
\draw node[below] at (-60:3) {$D$};
\end{tikzpicture}\\
\caption{Ptolemy transformation of $\dia ABCD$}
\label{ptolemy}
\end{figure}
  %%%%%%%%  %%%%%%%%  %%%%%%%%
%==============================================================================

\subsection{Even Ptolemy transformations}\label{subsec:even}
Since the $\l$-lengths are invariant under the action of $\SL(1|2)$, let us assume the quadrilateral in question is in standard position.

\begin{Thm}\label{bosonicptolemy}Assume $\dia ABCD\in \cS_{h_e}^{\h,\s}$ labeled as in Figure \ref{ptolemy}, and also $f^2:=\<B,D\>$. Then we have the (bosonic) Ptolemy relations
\Eq{\label{ef}
ef&=(ac+bd)\left(1+\frac{h_e\inv\s_1\h_2}{2(\sqrt\mathcal{X}+\sqrt\mathcal{X}\inv)}+\frac{h_e\s_2\h_1}{2(\sqrt\mathcal{X}+\sqrt\mathcal{X}\inv)}\right).
}

The Ptolemy relation is invariant under the change of representatives $\h,\s$ and the corresponding $h_e$.
\end{Thm}
\begin{proof} By definition we have the coordinates
\Eqn{
B&=t(1,1,1,\frac{\h_1\h_2}{2}|h_e\h_1,h_e\inv \h_2,-h_e\inv\h_2,h_e\h_1),\\
D&=(x_1,x_2,-y,z,\xi_1^+,\xi_2^+,\xi_1^-,\xi_2^-).
}
Writing all the variables in terms of $\l$-lengths, we have
\Eqn{
f^2=&\<B,D\>\\
=&_{\eqref{pairing}}t\left(\frac{1}{2}(x_1+x_2)+y+\frac{1}{2}(h_e\h_1\xi_1^--h_e\inv\h_2\xi_2^-+h_e\inv\h_2\xi_1^++h_e\h_1\xi_2^+)+\frac{z\h_1\h_2}{2}\right)\\
=&\frac{b^2d^2}{e^2}+\frac{a^2c^2}{e^2}+\frac{2abcd}{e^2}\\
&+\frac{abcd}{e^2}(h_e^2\sqrt\mathcal{X}\h_1\s_2-h_e^{-2}\sqrt\mathcal{X} \h_2\s_1-h_e^{-2}\sqrt\mathcal{X}\inv\h_2\s_1-h_e^2\sqrt\mathcal{X}\inv\h_1\s_2)+\frac{abcd}{2e^2}\s_1\s_2\h_1\h_2\\
&=\frac{(ac+bd)^2}{e^2}+\frac{abcd}{e^2}\left(h_e^{-2}(\sqrt\mathcal{X}+\sqrt\mathcal{X}\inv)\s_1\h_2+h_e^2(\sqrt\mathcal{X}+\sqrt\mathcal{X}\inv)\s_2\h_1+\half\s_1\s_2\h_1\h_2\right).
}
Hence
\Eqn{
e^2f^2=(ac+bd)^2+R,
}
where $$R=abcd\left(h_e^{-2}(\sqrt\mathcal{X}+\sqrt\mathcal{X}\inv)\s_1\h_2+h_e^2(\sqrt\mathcal{X}+\sqrt\mathcal{X}\inv)\s_2\h_1+\half\s_1\s_2\h_1\h_2\right).$$
Note that $$R^2=2(abcd)^2(\sqrt\mathcal{X}+\sqrt\mathcal{X}\inv)^2\s_1\s_2\h_1\h_2$$
and $R^n=0$ for $n\geq 3$. Also we have 
$$(ac+bd)^2 = (\sqrt\mathcal{X}+\sqrt\mathcal{X}\inv)^2(abcd).$$
Hence we have
\Eqn{
ef&=(ac+bd)\sqrt{1+\frac{R}{(ac+bd)^2}}\\
&=(ac+bd)\left(1+\frac{1}{2}\frac{R}{(ac+bd)^2}-\frac{1}{8}\frac{R^2}{(ac+bd)^4}\right)\\
&=(ac+bd)\left(1+\frac{h_e^{-2}(\sqrt\mathcal{X}+\sqrt\mathcal{X}\inv)\s_1\h_2+h_e^2(\sqrt\mathcal{X}+\sqrt\mathcal{X}\inv)\s_2\h_1+\half\s_1\s_2\h_1\h_2}{2(\sqrt\mathcal{X}+\sqrt\mathcal{X}\inv)^2}\right.\\
&\tab -\left.\frac{\s_1\s_2\h_1\h_2}{4(\sqrt\mathcal{X}+\sqrt\mathcal{X}\inv)^2}\right)\\
&=(ac+bd)\left(1+\frac{h_e^{-2}\s_1\h_2}{2(\sqrt\mathcal{X}+\sqrt\mathcal{X}\inv)}+\frac{h_e^2\s_2\h_1}{2(\sqrt\mathcal{X}+\sqrt\mathcal{X}\inv)}\right).
}

By Proposition \ref{ratiounique} for a different choice of representative $\h',\s'$, there exists $g'\in \cZ\sub \SL(1|2)$ such that $g'\cdot \dia ABCD \in \cS_{h_e'}^{\h',\s'}$. Hence formula \eqref{ef} remains invariant since the pairing is invariant under the $\cZ$-action.
\end{proof}

%==============================================================================

\subsection{Odd Ptolemy transformation}\label{subsec:odd}

Since the (equivalence classes of) odd invariants do not change under the $\SL(1|2)$ action, let us first calculate the Ptolemy action on fermions. Assume $\dia ABCD\in \cS_{h_e}^{\h,\s}$. We would like to bring $\dia DABC\in \cS_{h_f}^{\mu,\nu}$ by the $\SL(1|2)$ action and calculate the new odd parameters $\mu,\nu$ which comprise the odd Ptolemy transformation. On the other hand, with the new configuration we shall have new ratios $h_a', h_b', h_c', h_d'$ with the attaching 4 triangles outside $\dia DABC$. Recall that the lift is the same under vertex rescaling, so we have a freedom to rescale the fermions $\mu$ and $\nu$, which we shall do to fix the ratio $h_a'$ and $h_b'$ in the sequel that is defined up to certain constants according to Proposition \ref{ratiounique}.

Here we shall assume that no two sides of $\dia ABCD$ are identified with one other under the projection to the surface $\til{F}\to F$, so that the attaching triangles at the sides $a,b,c,d$ do not change under Ptolemy transformation of $\dia ABCD$. The special cases when sides are identified is treated in Appendix \ref{sec:specialPtolemy}.

We start with $\dia ABCD$ in standard position
\Eqn{
A&=r(0,1,0,0|0,0,0,0),\\
B&=t(1,1,1,\frac{\h_1\h_2}{2}|h_e\h_1,h_e\inv\h_2,-h_e\inv\h_2,h_e\h_1),\\
C&=s(1,0,0,0|0,0,0,0),\\
D&=(x_1,x_2,-y,z|\xi_1^+,\xi_2^+,\xi_1^-,\xi_2^-).
}
First using the prime transformation $P_{h_e,h_a}^\h$, we bring $\D BCA$ into standard position:
\Eqn{
A&=r(1,0,0,0|0,0,0,0),\\
B&=t(0,1,0,0|0,0,0,0),\\
C&=s(1,1,1,\frac{\h_1\h_2}{2}|h_a\h_1,h_a\inv\h_2,-h_a\inv\h_2,h_a\h_1),
}
while transforming $D$ by Lemma \ref{primeaction} into $({x_1}',{x_2}',{y}',{z}'|{\xi_1^+}',{\xi_2^+}',{\xi_1^-}',{\xi_2^-}')$,
where
\Eqn{
{x_2}'&=x_1,\\
{z}'&=z+\frac{1}{2}(h_e\inv\h_2\xi_1^+-h_e\h_1\xi_2^++\h_1\h_2 x_1),\\
{y}'&=y+\frac{1}{2}(h_e\inv\h_2\xi_1^++h_e\h_1\xi_2^+)+x_1,\\
{\xi_1^-}'&=c_\h h_a\inv(h_e\xi_2^+-\h_2x_1),\\
{\xi_2^-}'&=c_\h\inv h_a(-h_e\inv\xi_1^++\h_1x_1),
}
and ${x_1}',{\xi_1^+}'$, ${\xi_2^+}'$ can be recovered from \eqref{*}.

Now by a diagonal transform $D_a$ with $a=\sqrt[4]{\frac{x_2'}{x_1'}}$, we bring $\D BDA$ into standard position with odd parameters $(\til{\mu}_1, \til{\mu}_2)$ where
\Eqn{
\til{\mu}_1&=\frac{{\xi_2^-}'}{\sqrt{{x_2}'{y}'}}=\frac{h_a}{c_\h}\frac{-h_e\inv\xi_1^++\h_1x_1}{\sqrt{x_1(y+\frac{1}{2}(h_e\inv\h_2\xi_1^++h_e\h_1\xi_2^+)+x_1)}}.
}

Now rewriting using the cross-ratios from Lemma \ref{rst}, we find
\Eqn{x_1=y\mathcal{X}\inv,\tab \xi_1^+=-\sqrt\mathcal{X}\inv y \s_1,\tab \xi_2^+=-\sqrt\mathcal{X}\inv y \s_2}
 hence
\Eqn{
\til{\mu}_1&=\frac{h_a}{c_\h}\frac{h_e\inv\sqrt\mathcal{X}\inv y \s_1+\mathcal{X}\inv y\h_1}{\sqrt\mathcal{X}\inv\sqrt{y}\sqrt{y-\frac{1}{2}(h_e\inv\sqrt\mathcal{X}\inv y\h_2\s_1+h_e\sqrt\mathcal{X}\inv y\h_1\s_2)+\mathcal{X}\inv y}}\\
&=\frac{h_a}{h_ec_\h}\frac{h_e\h_1+\sqrt\mathcal{X} \s_1}{\sqrt{1+\frac{\sqrt\mathcal{X}}{2}(h_e\inv\s_1\h_2+h_e\s_2\h_1)+\mathcal{X}}}.
}
Similarly we have
\Eqn{
\til{\mu}_2=\frac{h_ec_\h}{h_a}\frac{ h_e\inv\h_2+\sqrt\mathcal{X} \s_2}{\sqrt{1+\frac{\sqrt\mathcal{X}}{2}(h_e\inv\s_1\h_2+h_e\s_2\h_1)+\mathcal{X}}}.
}

Let us denote the denominator by
\Eq{\label{cD}
\cD:=\sqrt{1+\frac{\sqrt\mathcal{X}}{2}(h_e\inv\s_1\h_2+h_e\s_2\h_1)+\mathcal{X}}}
 for simplicity.

\begin{Def}
we shall fix the choice of the odd parameters $\mu_1,\mu_2$ of the triangle $\D ABD$ to be
\Eq{\label{muABD}
\mu_1=\frac{h_e\h_1+\sqrt\mathcal{X} \s_1}{\cD},\tab \mu_2=\frac{h_e\inv\h_2+\sqrt\mathcal{X}\s_2}{\cD},
}
so that $(\til{\mu}_1,\til{\mu}_2)=(h_a'\mu_1,{h_a'}\inv\mu_2)$ with $h_a'=\frac{h_a}{h_e c_\h}$.
\end{Def}

Notice that according to Proposition \ref{ratiounique}, $h_a'$ is unique up to multiplication by a constant $c$ with $c\mu = \mu$. However, since we are working in $\til{C}(F,\D)$, we have a freedom to rescale the ratios. Hence we shall fix our choice here as $h_a'=\frac{h_a}{h_e c_\h}$.

Similarly, using instead the inverse prime transform $P_{h_e,h_b}^{\h,-}$ on $\D ABC$, and again applying a diagonal transformation, we bring $\D CDB$ into standard position with odd parameters $(\til{\nu}_1,\til{\nu}_2)$ where
\Eqn{
\til{\nu}_1=\frac{h_b c_\h}{h_e}\frac{\sqrt\mathcal{X} h_e \h_1-\s_1}{\cD},\tab \til{\nu}_2=\frac{h_e}{h_b c_\h}\frac{\sqrt\mathcal{X} h_e\inv \h_2-\s_2}{\cD}.
}

Now recall that we want the configuration $\dia DABC\in \cS_{h_f}^{\mu,\nu}$, hence according to the spin graph evolution rule from Figure \ref{flipgraph}, we need to change the signs with respect to the attaching triangle at $BC$.  \begin{Def}
we shall fix the choice of the odd parameters $\nu_1,\nu_2$ of the triangle $\D CDB$ to be
\Eq{\label{nuCDB}
\nu_1=\frac{\s_1-\sqrt\mathcal{X} h_e\h_1}{\cD},\tab \nu_2=\frac{\s_2-\sqrt\mathcal{X} h_e\inv \h_2}{\cD},
}
so that $(\til{\nu}_1,\til{\nu}_2)=(-h_b'\nu_1,-{h_b'}\inv\nu_2)$ with $h_b'=\frac{h_b c_\h}{h_e}$.
\end{Def}

Again there is a freedom of choice of $h_b'$ up to a multiplication by a constant $c$ with $c\nu = \nu$, and we still have a freedom to rescale the ratio at $\nu$. Hence we shall fix our choice here as $h_b'=\frac{h_b c_\h}{h_e}$. Now this fixes all the rescaling of the remaining variables.

\begin{Lem} We have
\Eqn{
c_\mu c_\nu = c_\h c_\s.
}
\end{Lem}
\begin{proof} Expanding in terms of $\h$ and $\s$ and simplifying, we get
\Eqn{
c_\mu c_\nu&=\left(1+\frac{\mu_1\mu_2}{6}\right)\left(1+\frac{\nu_1\nu_2}{6}\right)\\
&=\left(1+\frac{(1+\mathcal{X})\h_1\h_2}{6\cD^2}\right)\left(1+\frac{(1+\mathcal{X})\s_1\s_2}{6\cD^2}\right).
}
Note that
\Eqn{
\frac{1+\mathcal{X}}{\cD^2} &= \frac{1+\mathcal{X}}{1+\mathcal{X}+\frac{\sqrt\mathcal{X}}{2}(h_e\s_2\h_1+h_e\inv \s_1\h_2)}=1+P(\s_2\h_1, \s_1\h_2)
}
for some polynomials $P$. Hence 
$$\frac{1+\mathcal{X}}{\cD^2}\h_1\h_2=\h_1\h_2, \tab \frac{1+\mathcal{X}}{\cD^2}\s_1\s_2 = \s_1\s_2,$$
and we have
$$c_\mu c_\nu = c_\h c_\s.$$
\end{proof}

Now we are in a position to compare the two lifts under a Ptolemy transform. We consider 4 attaching triangles of $\dia ABCD$ as in Figure \ref{8points} in the universal covering $\til{F}$. Note that the two configurations share the same triangle $\D ABD_1$ attached to $\D ABC$ and $\D ABD$ respectively.

%%%%%%%%%    NICE 8 POINTS PICTURE   %%%%%%%%%%%%%%%%
\begin{figure}[h!]
\centering
\begin{tikzpicture}[baseline,ultra thick]
\draw (-2,0)--(2,0);
\begin{scope}[bend right]
\draw (-2,0)to(0,2)to(2,0)to(0,-2)to(-2,0);
\draw (-2,0)to(135:2)to(0,2)to(45:2)to(2,0)to(-45:2)to(0,-2)to(-135:2)to(-2,0);
\end{scope}
\draw node[left] at (-2,0) {$A$};
\draw node[above] at (0,2) {$B$};
\draw node[right] at (2,0) {$C$};
\draw node[below] at (0,-2) {$D$};
\draw node[above left] at (135:2) {$D_1$};
\draw node[above right] at (45:2) {$D_2$};
\draw node[below right] at (-45:2) {$D_3$};
\draw node[below left] at (-135:2) {$D_4$};
\draw node at (0,0.5) {$\h$};
\draw node at (0,-0.5) {$\s$};
\draw node at (-1,1) {$\s_1$};
\draw node at (1,1) {$\s_2$};
\draw node at (1,-1) {$\s_3$};
\draw node at (-1,-1) {$\s_4$};
\end{tikzpicture}
\begin{tikzpicture}[baseline]
\draw[<->, thick](0,0)--(1,0);
\node[above]  at (0.5,0) {flip};
\end{tikzpicture}
\begin{tikzpicture}[baseline,ultra thick]
\draw (0,2)--(0,-2);
\begin{scope}[bend right]
\draw (-2,0)to(0,2)to(2,0)to(0,-2)to(-2,0);
\draw (-2,0)to(135:2)to(0,2)to(45:2)to(2,0)to(-45:2)to(0,-2)to(-135:2)to(-2,0);
\end{scope}
\draw node[left] at (-2,0) {$A$};
\draw node[above] at (0,2) {$B$};
\draw node[right] at (2,0) {$C$};
\draw node[below] at (0,-2) {$D$};
\draw node[above left] at (135:2) {$D_1$};
\draw node[above right] at (45:2) {$D_2$};
\draw node[below right] at (-45:2) {$D_3$};
\draw node[below left] at (-135:2) {$D_4$};
\draw node at (0.5,0) {$\nu$};
\draw node at (-0.5,0) {$\mu$};
\draw node at (-1,1) {$\s_1$};
\draw node at (1,1) {$\s_2$};
\draw node at (1,-1) {$\s_3$};
\draw node at (-1,-1) {$\s_4$};
\end{tikzpicture}
\caption{Labelling of the 4 attached triangles}
\label{8points}
\end{figure}
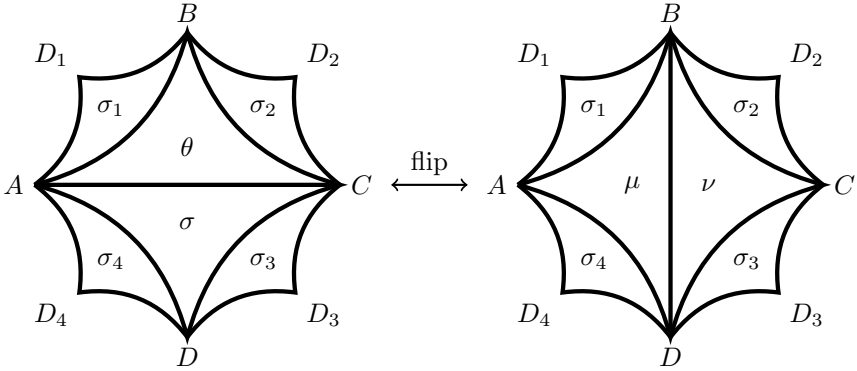
%%%%%%%%%%%%%%%%%%%%%%%%%

In particular, consider the lift of the two configurations with $\D AD_1B$ as the image of the base triangle. Then by definition the lifts differ by a post composition with $Z_{\frac{h_a'}{h_a}}$. Since the Ptolemy transform is local, $h_a'$  depend on $\s_1$, and this uniquely determines the relation between the two lifts. Then all the points on $\cL_0$ are uniquely defined by the lifts and also differ by $Z_{\frac{h_a'}{h_a}}$ for the image of both lifts. 

Let us denote by $\mathcal{X}_i$ the cross-ratio of the outside triangle containing $D_i$ with the attaching triangle in the quadrilateral before the flip, and $\mathcal{X}_i'$ the corresponding cross-ratio after the flip, with the common edge as the diagonal. For example. $\mathcal{X}_1, \mathcal{X}_1'$  are the cross-ratios of $\dia AD_1BC, \dia AD_1BD$ respectively.

Assume the spin structures $w_\sigma$ and $w_{\iota}$ coincide for simplicity. Also assume the lift of a point $X$ is defined by admissible transformation where the path goes through $\D B D_2 C$. Then by definition the image of the point $X$ is defined by the admissible transformations of the form
$$ g\circ P_{h_b\inv, h_s}^{\s_2,\pm}\circ\Up_{h_b}^{\mathcal{X}_2}\circ P_{h_a h_b}^\h\circ\Up_{h_a\inv}^{\mathcal{X}_1}$$
for the original lift with some $g$ defined outside of these triangles, and
$$ g\circ P_{{h_b'}\inv, h_s}^{\s_2,\pm}\circ Z_{-1}\circ \Up_{h_b'}^{\mathcal{X}_2'}\circ P_{h_f\inv, h_b'}^{\nu}\circ\Up_{h_f}^{\mathcal{X}\inv}\circ P_{h_a' h_f}^\mu\circ\Up_{{h_a'}\inv}^{\mathcal{X}_1'}$$
for the lift with flipped diagonal, where $g$ is the same as before and the $Z_{-1}$ term comes from the spin graph evolution. Since the lift differs by $Z_{\frac{h_a'}{h_a}}$, we have the equality:
\begin{eqnarray}
&&g\circ P_{h_b\inv, h_s}^{\s_2}\circ\Up_{h_b}^{\mathcal{X}_2}\circ P_{h_a h_b}^\h\circ\Up_{h_a\inv}^{\mathcal{X}_1}=\nonumber\\
&&g\circ P_{{h_b'}\inv, h_s}^{\s_2}\circ Z_{-1}\circ \Up_{h_b'}^{\mathcal{X}_2'}\circ P_{h_f\inv, h_b'}^{\nu}\circ\Up_{h_f}^{\mathcal{X}\inv}\circ P_{h_a' h_f}^\mu\circ\Up_{{h_a'}\inv}^{\mathcal{X}_1'}\circ Z_{\frac{h_a'}{h_a}}.\nonumber
\end{eqnarray}
Now recall that the above admissible transformations are compositions of
\Eqn{
\Up_{h_k'}^{\mathcal{X}_k} \circ P_{h_k,h_k'}^{\h_k}=(J D_{\sqrt\mathcal{X}_k} Z_{h_k'})\circ (Z_{c_{\h_k}}Z_{h_k'}\inv J \cU_{\h_k}\inv Z_{h_k})= D_{\sqrt\mathcal{X}_k\inv} Z_{-c_{\h_k}} \cU_{\h_k}\inv Z_{h_k},
}
where $\cU_{\h_k}$ are lower Borel subgroup elements with $1$'s on the diagonal. In particular, we can commute all the Borel elements to one side, again with $1$'s on the diagonals. Hence the diagonal of the product of admissible transformation depends only on the diagonal transformations $D, Z$, and the number of $J$'s changing the signs, which is taken care of automatically by the choice of the spin graph evolution. In particular, by looking at the diagonal entry and using the fact that the $\cZ$-subgroup commutes with $J$ and $D$, we immediately obtain the following relation:
\Eqn{
Z_{h_b}\inv D_{\sqrt\mathcal{X}_2\inv}Z_{c_\h}D_{\sqrt\mathcal{X}_1\inv}=Z_{h_b'}\inv D_{\sqrt{\mathcal{X}_2'}\inv} Z_{c_\nu} Z_{h_f\inv} D_{\sqrt\mathcal{X}} Z_{c_\mu} D_{\sqrt{\mathcal{X}'}\inv}Z_{\frac{h_a'}{h_a}}.
}
Using the fact that $\mathcal{X}_2\mathcal{X}_1=\mathcal{X}_2'\mathcal{X}\inv \mathcal{X}_1'$, we obtain 
$$\frac{h_b'}{h_b} = \frac{h_a'}{h_a} \frac{c_\nu c_\mu}{h_f c_\h}=\frac{h_a'}{h_a}\frac{c_\s}{h_f}.$$
Since $h_a'=\frac{h_a}{h_e c_\h}$ and $h_b'=\frac{h_b c_\h}{h_e}$, we get $h_f =\frac{c_\s}{c_\h^2}$. Using $\mathcal{X}_1'=\mathcal{X}_1 \cD$ which can be derived from \eqref{ef}, we can check that each off-diagonal entries is also the unity.

Using this trick, it is easy to obtain the other relations by just looking at the diagonal entries. We obtain the relations
\Eqn{
\frac{h_d'}{h_d}=\frac{h_a'}{h_a}h_e c_\nu,\tab \frac{h_c'}{h_c}=\frac{h_b'}{h_b}\frac{h_e}{c_\mu},
}
and finally get the second main result of this paper:
\begin{Thm}\label{oddptolemy}
The Ptolemy transformation is given by
\Eq{
\mu_1=\frac{h_e\h_1+\sqrt\mathcal{X} \s_1}{\cD},\tab\mu_2=\frac{h_e\inv\h_2+\sqrt\mathcal{X} \s_2}{\cD} \label{mu}\\
\nu_1=\frac{\s_1-\sqrt\mathcal{X} h_e\h_1}{\cD},\tab \nu_2=\frac{\s_2-\sqrt\mathcal{X} h_e\inv\h_2}{\cD}\label{nu}
}
\Eq{
h_a' = \frac{h_a}{h_e c_\h},\tab h_b' = \frac{h_b c_\h}{h_e},\tab h_c' =h_c\frac{c_\h}{c_\mu}, \tab h_d'=h_d\frac{c_\nu}{c_\h},\tab h_f =\frac{c_\s}{c_\h^2},
}
where $$\cD=\sqrt{1+\mathcal{X}+\frac{\sqrt\mathcal{X}}{2}(h_e\inv\s_1\h_2+h_e\s_2\h_1)}$$ and $c_{\theta}=1+\frac{\theta_1\theta_2}{6}$, while the signs of the fermions follow the spin graph evolution rule as in Figure \ref{flipgraph}.
\end{Thm}

%After flipping twice, we arrive at the same equivalence class of coordinate under vertex rescaling, in other words 
%$$(h_a'', h_b'', h_c'', h_d'', h_e'' | \h'', \s'') = 
%(h_a \a_1, h_b \a_1, h_c \a_2, h_d \a_2, h_e\inv \a_1\inv\a_2 | \a_2\inv \s,-\a_1\inv \h)_{upsidedown},$$
%where $\a_1=\frac{c_\h^2}{h_e c_\nu}, \a_2=\frac{c_\h^2}{c_\mu^2}$, and the new coordinate corresponds to the triangulation where the quadrilateral is flipped, which changes the numeration. In particular by another upside-down transformation, we get back the same configuration.

After flipping twice, we arrive at the same equivalence class of coordinate under vertex rescaling, or in other words 
$$(h_a'', h_b'', h_c'', h_d'', h_e'' | \h'', \s'') = 
(h_a \a_1, h_b \a_1, h_c \a_2, h_d \a_2, h_e\inv \a_1\inv\a_2 | \a_2\inv \s,-\a_1\inv \h),$$
where double prime denote the variables after application of Ptolemy transformation twice,  
$\a_1=\frac{c_\h^2}{h_e c_\nu}, \a_2=\frac{c_\h^2}{c_\mu^2}$, and the new coordinate tuple corresponds to the triangulation where the quadrilateral is turned upside down, altering the numeration. 
In particular by an upside-down transformation, we get back the same configuration.

In the cases $F=F_1^1,F_0^3$ when some of the sides of $\dia ABCD$ are identified, the Ptolemy transform of the ratios can be calculated by looking at the representation of $\what\rho(\c)$ directly as in Appendix \ref{sec:specialPtolemy}.

\begin{Cor} If two triangulations $\D$ and $\D'$ with fixed representative of the spin graphs are related by a finite sequence of flips, then the super Ptolemy transformations in Theorem \ref{bosonicptolemy} and Theorem \ref{oddptolemy} provide a coordinate transformation from $C(F,\D)$ to $C(F,\D')$. In particular the mapping class group $MC(F)$ acts naturally on $S\til{T}(F)$ permuting the spin components.
\end{Cor}
\begin{proof} The argument for the mapping class group action follows as in \cite{PZ}.
\end{proof}
%==============================================================================
\appendix
\section{$SL(1|2)$: Notations and conventions}\label{sec:SL12}
Following \cite{PZ}, we shall be working over the Grassmann algebra $S_\R$ over the field $\R$ with possibly infinitely many generators and with standard decomposition $S_\R=S_0\o+S_1$ into even and odd elements. Given $a\in S_\R$ there is a projection $S_\R\to \R$ to its constant term $a\mapsto a_\#$ called the \emph{body} of $a$. By abuse of notation, unless otherwise specified we shall write $a\neq 0$ if the body of $a$ is non-zero and $a>0$ if the body of $a$ is positive. Then $a$ is invertible if and only if $a\neq 0$. For simplicity we write $\R_+$ for the subspace of $S_0$ with positive bodies.
\begin{Def} The super vector space $\R^{m|n}$ is defined to be the space
\Eq{
\R^{m|n} := \{(z_1,z_2,...,z_m|\h_1,\h_2,...,\h_n) : z_i\in S_0, \h_j\in S_1\},
}
and 
$\R_+^{m|n}$ a subspace of $\R^{m|n}$ such that all even coordinates $z_i$ have positive bodies.
\end{Def}

Let $\g$ be a Lie superalgebra, one can consider its Grassmann envelope $\g(S_\R):=S_\R\ox \g$ so that one can construct a representation of $\g(S_\R)$ in the space $S_\R\ox \R^{m|n}$ from a given representation of $\g$ in $\R^{m|n}$. One can then produce a representation of the corresponding Lie supergroup $G(S_\R)$ by exponentiating pure even elements from $\g(S_\R)$ in $S_\R\ox \R^{m|n}$.

In this paper we shall only consider matrix representations on $\R^{2|1}$. For conventional convenience we shall choose the basis of $\R^{2|1}$ such that the matrix elements of a pure even matrix are of the form
\Eq{
\veca{b&f&b\\f&b&f\\b&f&b},
}
where $b$ and $f$ stand for even (bosonic) and odd (fermionic) parity respectively. The supermatrix multiplication is then given by
\Eq{\label{signs}
&\veca{a&\c&b\\\a&f&\d\\c&\b&d}\veca{a'&\c'&b'\\\a'&f'&\d'\\c'&\b'&d'}:=
\veca{
aa'-\c\a'+bc'&a\c'+\c f'+b\b'&ab'-\c\d'+bd'\\
\a a'+f\a'+\d c'& -\a\c'+ff'-\d\b'&\a b'+f\d'+\d d'\\
ca'-\b\a'+dc'&c\c'+\b f'+d\b'&cb'-\b \d'+dd'
}.
}

Note that our convention has the signs of the products of odd coordinates reversed. This follows from the choice of notation that 
$$(\a\mu)(\b\nu)=(-\a\b)(\mu\nu)$$
in the universal enveloping algebra if $\a,\b$ are odd elements of $\cS_\R$ and $\mu,\nu$ are odd elements of the Lie superalgebra. 

\begin{Def} The superdeterminant or Berezinian of an even matrix $M=\veca{A&B\\C&D}$ with $D$ invertible is defined as
\Eq{
sdet(M)=(\det D)\inv \det(A+BD\inv C).
}
The supertrace of a supermatrix is given by
\Eqn{
str\veca{a&*&*\\ *&b&*\\ *&*&c}:=a+c-b.
}
\end{Def}
Using our convention, the superdeterminant of a diagonal matrix is given by
\Eq{
sdet\veca{a&0&0\\0&f&0\\0&0&d} = \frac{ad}{f}.
}
\begin{Def}
$\sl(1|2)\simeq \mathfrak{osp}(2|2)$ is the Lie superalgebra spanned by four odd generators $e_1^\pm, e_2^\pm$ and four bosonic generators $E,F,h_1,h_2$. Elements $e_1^\pm, e_2^\pm, h_{1}, h_{2}$ are the Chevalley generators and satisfy the relations of $A(0,1)$ superalgebra where both simple roots are chosen to be grey in Kac's classification of Lie superalgebra (\cite{dict, kac}). One can find explicit commutation relations in \cite{GQS}. 
In particular it can be realized as the algebra of supertraceless $(2|1)\x(2|1)$ supermatrices so that
\Eqn{
h_1&=\veca{1&0&0\\0&1&0\\0&0&0},\tab h_2=\veca{0&0&0\\0&1&0\\0&0&1},\\
e_1^+&=\veca{0&1&0\\0&0&0\\0&0&0},\tab e_2^+=\veca{0&0&0\\0&0&1\\0&0&0},\\
e_1^-&=\veca{0&0&0\\1&0&0\\0&0&0},\tab e_2^-=\veca{0&0&0\\0&0&0\\0&1&0},\\
E&=\veca{0&0&1\\0&0&0\\0&0&0},\tab F=\veca{0&0&0\\0&0&0\\1&0&0}.
}
Let us also define
\Eqn{
H:=h_1-h_2=\veca{1&0&0\\0&0&0\\0&0&-1},\tab Z:=h_1+h_2=\veca{1&0&0\\0&2&0\\0&0&1}.
}
Then $\{H,E,F\}$ is an $\sl(2)$-triple, $\{h_i,e_i^+,e_i^-\}_{i=1,2}$ are $\gl(1|1)$ triples, and $Z$ generates the subgroup $\cZ$ that commutes with the bosonic generators.
\end{Def}

The corresponding supergroup $SL(1|2)$ can be faithfully realized as $(2|1)\x(2|1)$ supermatrices with $sdet(g)=1$. In particular a diagonal matrix $\veca{a&0&0\\0&f&0\\0&0&d}$ belongs to $SL(1|2)$ if and only if $ad=f$.

\begin{Prop}\label{involution}
There exists an involution $\Psi$ given by
\Eqn{
e_1^+\mapsto e_2^+,&\tab e_1^-\mapsto -e_2^-,\\
e_2^+\mapsto e_1^+,&\tab e_2^-\mapsto -e_1^-,\\
h_1\mapsto -h_2,&\tab h_2\mapsto -h_1,\\
E\mapsto E,&\tab F\mapsto F, \tab H\mapsto H.
}
It preserves the $\sl(2)$ triple, and interchanges the two $\gl(1|1)$ triples.
\end{Prop}
\begin{Def}\label{SL0} we shall denote by
\Eq{\til{SL}(1|2):=\Psi \ltimes SL(1|2)_0.}
where $SL(1|2)_0$ is the component of $SL(1|2)$ with $f>0$. In particular, there is a canonical projection
\Eqn{\SL(1|2)\to SL(2,\R)}
given by sending $\Psi\to Id_{2\x 2}$ and $\veca{a&*&b\\{*}&f&*\\c&*&d}\to \frac{1}{\sqrt{f_\#}}\veca{a_\#&b_\#\\c_\#&d_\#}$.
\end{Def}

\begin{Rem}\label{SL12choice}
In principle, instead of taking $\widetilde{SL}(1|2)$, we could consider, from the group-theoretic point of view,  a more natural embedding of $\widetilde{SL}(1|2)$ into $OSp(2|2)$, the supergroup of $(2|2)\times(2|2)$ supermatrices preserving a certain bilinear form, so that the involution $\Psi$ is described by a certain $(2|2)\times(2|2)$ supermatrix. However, since we are only working with the \emph{adjoint} representation of the corresponding group in the light cone formalism, it is much easier to work with $(1|2)\times(1|2)$ supermatrices, and that explains our choice of embedding.

One can recover that $(2|2)\times(2|2)$ realization of $\widetilde{SL}(1|2)$ from the action of $\widetilde{SL}(1|2)$ on the rays of the light cone $\mathcal{L}_0$, which coincides with the action of fractional linear transformations on $\mathbb{R}^{1|2}$ (see e.g., \cite{Nat}) realized as a subgroup of $OSp(2|2)$. 
This gives the relation of our constructions with the geometry of uniformized $\mathcal{N}=2$ Riemann surfaces \cite{cohn, drs, br}.
\end{Rem}

%==============================================================================

\section{Ptolemy transformation in the special cases}\label{sec:specialPtolemy}
In this section, we shall deal with the Ptolemy transformation of the ratio coordinates where some of the sides of the quadrilateral $\dia ABCD$ are identified in the projection $\til{F}\to F$ to the surface. There are basically two cases: where two consecutive sides are identified or two opposite sides are identified. Similar to the argument in the general case, it suffices to look at the admissible transformation relating the sides. In particular, these are exactly the representations of the simple cycle $\c\in\pi_1$ that joins the corresponding edge. We note that $\mu, \nu$ are given by the same formula as before, and by fixing the rescaling, we can also let $h_f=\frac{c_\s}{c_\h^2}$ be the same expression as before.

In all the cases, we let $\D BCA$ be the base triangle of the lift before the flip, and $\D BDA$ be the base triangle after the flip. Then by assumption these two lifts are related by the post-composition with a diagonal transformation $D_{\sqrt[4]{\frac{x_2'}{x_1'}}}=D_{\sqrt\cD}$.

%%%%%%%    COMPLICATED PICTURE  %%%%%%%%%%%

\begin{figure}[h!]
\centering
\begin{tikzpicture}[baseline]
\draw (-2,0)--(2,0);
  \draw[ 
       very thick, bend right,
        decoration={markings, mark=at position 0.65 with {\arrow{>}}},
        postaction={decorate}
        ]
        (-2,0) to (0,2);
  \draw[ 
        very thick, bend right,
        decoration={markings, mark=at position 0.4  with {\arrow{<}}},
        postaction={decorate}
        ]
        (0,2) to (2,0);
\begin{scope}[bend right]
\draw (2,0)to(0,-2)to (-2,0);
\draw (0,2)to(45:2)to(2,0);
\draw (45:2)to(22.5:2)to(2,0);
\end{scope}
\draw node[left] at (-2,0) {$A$};
\draw node[above] at (0,2) {$B$};
\draw node[right] at (2,0) {$C$};
\draw node[below] at (0,-2) {$D$};
\draw node[above right] at (45:2) {$D_2$};
\draw node at (-0.7,0.3) {$h_a$};
\draw node at (0.7,0.3) {$h_a\inv$};
\draw node at (-0.5,1.3) {$a$};
\draw node at (0.5,1.3) {$a$};
\draw node at (0.7,1.6) {$a$};
\draw node at (0,-0.2) {$e$};
\draw node at (1.3,0.7) {$e$};
\draw node at (0.5,-1.3) {$c$};
\draw node at (-0.5,-1.3) {$d$};
\draw[ultra thick] (-1,1)--(0,0.5)--(1,1);
\end{tikzpicture}
\begin{tikzpicture}[baseline]
\draw[<->, thick](0,0)--(1,0);
\node[above]  at (0.5,0) {flip};
\end{tikzpicture}
\begin{tikzpicture}[baseline]
\draw (0,2)--(0,-2);
  \draw[ 
        very thick, bend right,
        decoration={markings, mark=at position 0.65 with {\arrow{>}}},
        postaction={decorate}
        ]
        (-2,0) to (0,2);
  \draw[ 
        very thick, bend right,
        decoration={markings, mark=at position 0.4 with {\arrow{<}}},
        postaction={decorate}
        ]
        (0,2) to (2,0);
\begin{scope}[bend right]
\draw (2,0)to(0,-2)to (-2,0);
\draw (0,2)to(45:2)to(22.5:2);
\draw (0,2)to(22.5:2)to(2,0);
\end{scope}
\draw node[left] at (-2,0) {$A$};
\draw node[above] at (0,2) {$B$};
\draw node[right] at (2,0) {$C$};
\draw node[below] at (0,-2) {$D$};
\draw node[above right] at (45:2) {$D_2$};
\draw[ultra thick] (-1,1)--(-0.5,0)--(0.5,0)--(1,0.8);
\draw node at (-1,0) {$h_a'$};
\draw node at (1,0) {${h_a'}\inv$};
\draw node at (-0.5,1.3) {$a$};
\draw node at (0.2,1.1) {$a$};
\draw node at (1,1.2) {$f$};
\draw node at (-0.2,0.5) {$f$};
\draw node at (2,0.5) {$d$};
\draw node at (0.5,-1.3) {$c$};
\draw node at (-0.5,-1.3) {$d$};
\end{tikzpicture}
\caption{Case 1: Consecutive sides identified plus the path of admissible transformation}
\label{case1}
\end{figure}

%%%%%%%  %%%%%%%%%%%

\begin{figure}[h!]
\centering
\begin{tikzpicture}[baseline]
\draw (-2,0)--(2,0);
  \draw[ 
        very thick, bend right,
        decoration={markings, mark=at position 0.65 with {\arrow{>}}},
        postaction={decorate}
        ]
        (-2,0) to (0,2);
  \draw[ 
        very thick, bend right,
        decoration={markings, mark=at position 0.65  with {\arrow{<}}},
        postaction={decorate}
        ]
        (2,0) to (0,-2);
\begin{scope}[bend right]
\draw (0,2)to(2,0);
\draw (2,0)to(0,-2)to (-2,0);
\draw (2,0)to(-45:2)to(-67.5:2)to(0,-2);
\draw (-45:2)to(0,-2);
\end{scope}
\draw node[left] at (-2,0) {$A$};
\draw node[above] at (0,2) {$B$};
\draw node[right] at (2,0) {$C$};
\draw node[below] at (0,-2) {$D$};
\draw node[below right] at (-45:2) {$D_3$};
\draw node at (-0.7,0.3) {$h_a$};
\draw node at (0.7,-0.3) {$h_a\inv$};
\draw node at (-0.5,1.3) {$a$};
\draw node at (0.5,1.3) {$b$};
\draw node at (0.2,0.2) {$e$};
\draw node at (0.5,-1.3) {$a$};
\draw node at (-0.5,-1.3) {$d$};
\draw node at (1.6,-0.7) {$b$};
\draw node at (1,-1.3) {$e$};
\draw[ultra thick] (-1,1)--(0,0.5)--(0,-0.5)--(1,-1);
\end{tikzpicture}
\begin{tikzpicture}[baseline]
\draw[<->, thick](0,0)--(1,0);
\node[above]  at (0.5,0) {flip};
\end{tikzpicture}
\begin{tikzpicture}[baseline]
\draw (0,2)--(0,-2);
  \draw[ 
        very thick, bend right,
        decoration={markings, mark=at position 0.65 with {\arrow{>}}},
        postaction={decorate}
        ]
        (-2,0) to (0,2);
  \draw[ 
       very thick, bend right,
        decoration={markings, mark=at position 0.65 with {\arrow{<}}},
        postaction={decorate}
        ]
        (2,0) to (0,-2);
\begin{scope}[bend right]
\draw (0,2)to(2,0);
\draw (2,0)to(0,-2)to (-2,0);
\draw (2,0)to(-45:2)to(-67.5:2)to(0,-2);
\draw (2,0)to(-67.5:2);
\end{scope}
\draw node[left] at (-2,0) {$A$};
\draw node[above] at (0,2) {$B$};
\draw node[right] at (2,0) {$C$};
\draw node[below] at (0,-2) {$D$};
\draw node[below right] at (-45:2) {$D_3$};
\draw node at (-1,0.2) {$h_a'$};
\draw node at (1.1,-0.2) {${h_a'}\inv$};
\draw node at (-0.5,1.3) {$a$};
\draw node at (0.5,1.3) {$b$};
\draw node at (0.2,0.2) {$f$};
\draw node at (0.5,-1.3) {$a$};
\draw node at (-0.5,-1.3) {$d$};
\draw node at (1,-1.3) {$f$};
\draw node at (0.4,-2) {$d$};
\draw[ultra thick] (-1,1)--(-0.5,0)--(0.5,0)--(0.9,-0.9);
\end{tikzpicture}
\caption{Case 2: Opposite sides identified plus the path of admissible transformation}
\label{case2}
\end{figure}
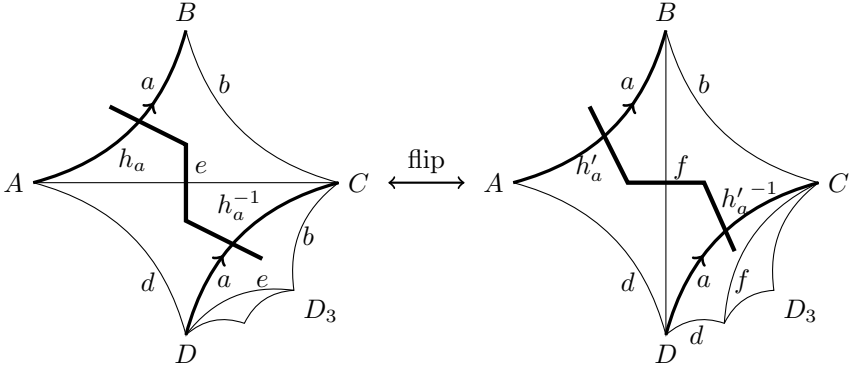

%%%%%%%  %%%%%%%%%%%

Case 1: Two adjacent sides. Assume $\vec[AB]$ and $\vec[CB]$ are identified. Let $g_1$ and $g_2$ be the representations of the cycles joining the edge before and after the lift. Then
$$g_1=\Up_{h_a\inv}^{\mathcal{X}_2}\circ P_{h_a h_a\inv}^{\h}$$ 
and
$$g_2=Z_{-1}\circ \Up_{{h_a'}\inv}^{\mathcal{X}_2'}\circ P_{h_f\inv {h_a'}\inv}^{\nu}\circ\Up_{h_f}^{\mathcal{X}\inv}\circ P_{h_a' h_f}^{\mu}$$
with $g_1 = D_{\sqrt\cD}\circ g_2 \circ D_{\sqrt\cD}\inv$.
Note that we have two $J$'s in $g_1$ and four $J$'s in $g_2$, hence we see that the sign $Z_{-1}$ is needed for the conjugation. In particular the orientation of the branch corresponding to $h_a$ is flipped in the spin graph evolution.

Since we only have the (clockwise) prime transformation $P^{+}$, the same argument as in the general case shows that we only need to look at the diagonal entries. Using $\mathcal{X}\inv \mathcal{X}_2' = \mathcal{X}_2=1$, we get
$$Z_{c_\h}Z_{h_a} =  Z_{c_\nu}Z_{h_f}\inv Z_{c_\mu}Z_{h_a'}$$
or $$h_a' = \frac{c_\h}{c_\mu c_\nu}h_f h_a =\frac{h_f}{c_\s}h_a= \frac{h_a}{c_\h^2}.$$

If $\vec[BC]$ and $\vec[DC]$ are identified instead, the inverse argument shows that the rule is given by
$$h_b' = \frac{c_\mu}{h_e}h_b,$$ and again an extra sign is needed which determines the spin graph evolution.

\begin{Cor} In the case of $F_0^3$, where $\vec[AB]$ is identified with $\vec[CB]$, and $\vec[BC]$ is identified with $\vec[DC]$,  the Ptolemy transform $\R^{6|4}\to \R^{6|4}$ is given by
$$(a,b,e,h_a, h_b, h_e |\h_1,\h_2, \s_1,\s_2)\mapsto (a,b,f,\frac{h_a}{c_\h^2}, \frac{c_\mu}{h_e} h_b, \frac{c_\s}{c_\h^2}| \mu_1,\mu_2, \nu_1,\nu_2),$$
where $f$ is given by $\eqref{ef}$, $\mu,\nu$ are given by \eqref{mu}-\eqref{nu}, and the spin graph evolves as in Figure \ref{flipgraph} but flipping all four outside branches. Furthermore, flipping twice return to the same configuration.
\end{Cor}
Case 2: Two opposite sides. Assume $\vec[AB]$ and $\vec[DC]$ are identified. Then
$$g_1 =\Up_{h_a\inv}^{\mathcal{X}_3}\circ P_{h_e\inv, h_a\inv}^{\s,+}\circ \Up_{h_e}^{\mathcal{X}}\circ P_{h_a, h_e}^{\h,-}$$
and
$$g_2= \Up_{{h_a'}\inv}^{\mathcal{X}_3'}\circ P_{h_f\inv, {h_a'}\inv}^{\nu,-}\circ\Up_{h_f}^{\mathcal{X}\inv}\circ P_{h_a', h_f}^{\mu,+},$$
where again $g_1 = D_{\sqrt\cD}\circ g_2\circ D_{\sqrt\cD}\inv$. This time however, we see that ignoring the diagonal elements, the Borel parts are of the form
$$g_1\sim  \cU\circ J \circ \cU \circ J\inv,\tab g_2 \sim Z_{-1}\circ J\circ \cU \circ J \circ \cU.$$
In particular we see that no extra signs are needed, i.e., the spin graph evolution stays the same on the outside branches. However, the same argument as in the general case does not work here since the diagonal entries are more complicated. However, by commuting the Borel part across the $J$'s, they are of the form $g_1 \sim \cU \circ \cV$ and $g_2\sim \cV\circ \cU$, where $\cV$ is upper Borel. Hence we only need to look at the off diagonal even entry (i.e., top right corner) instead, which is very simple.

By carefully expanding the matrix, the top right entries are given by
\Eqn{
(g_1)_{1,3}&=\frac{c_\s h_a c_\h^2}{h_e \sqrt{\mathcal{X}_3\mathcal{X}}},\\
(D_{\sqrt\cD}\circ g_2\circ D_{\sqrt\cD}\inv)_{1,3}&=c_\nu^2 c_\mu\frac{h_a'}{h_f}\frac{\cD}{\sqrt{\mathcal{X}_3'\mathcal{X}}}.
}
Now using the fact that $\mathcal{X}_3=\frac{bd}{e^2}, \mathcal{X}_3'=\frac{f^2}{bd}$ and $\frac{ef}{bd}=\cD$ so that $\sqrt\mathcal{X}_3' = \sqrt\mathcal{X}_3\cD$, we get
$$\frac{c_\s h_a c_\h^2}{h_e}=c_\nu^2 c_\mu\frac{h_a'}{h_f}.$$
With $h_f=\frac{c_\s}{c_\h^2}$ and $c_\h c_\s=c_\mu c_\nu$ we finally obtain the relation
$$h_a'=\frac{h_f}{h_e} \frac{c_\h}{c_\nu}h_a=\frac{c_\mu}{c_\h^2 }\frac{h_a}{h_e}.$$
Again if $\vec[BC]$ and $\vec[AD]$ are identified, a similar argument shows that the rule is given by
$$h_b'=\frac{h_b}{h_e h_f}\frac{c_\s}{c_\nu}=\frac{c_\h^2}{c_\nu}\frac{h_b}{h_e},$$
and no extra signs are needed.

\begin{Cor} In the case of $F_1^1$, where $\vec[AB]$ is identified with $\vec[DC]$ and $\vec[BC]$ is identified with $\vec[AD]$, the Ptolemy transform $\R^{6|4}\to \R^{6|4}$ is given by
$$(a,b,e,h_a, h_b, h_e |\h_1,\h_2, \s_1,\s_2)\mapsto (a,b,f,\frac{c_\mu}{c_\h^2}\frac{h_a}{h_e}, \frac{c_\h^2}{c_\nu}\frac{h_b}{h_e},\frac{c_\s}{c_\h^2}| \mu_1,\mu_2, \nu_1,\nu_2),$$
where $f$ is given by $\eqref{ef}$, $\mu,\nu$ are given by \eqref{mu}-\eqref{nu}, and the spin graph evolves as in Figure \ref{flipgraph} but all four outside branches are unflipped. Furthermore, flipping twice return to the same configuration.
\end{Cor}
%==============================================================================

\end{document}